\documentclass[11pt,a4paper]{amsart}
\usepackage[english]{babel}
\usepackage[utf8]{inputenc}  
\usepackage[T1]{fontenc}  
\usepackage{amscd,amssymb,amsmath,latexsym,url,mathrsfs,graphicx,upgreek}
\usepackage{hyperref} \hypersetup{breaklinks=true}

\newcommand{\abs}[1]{| #1 |}
\newcommand{\norm}[1]{\| #1 \|}
\newcommand{\Abs}[1]{\left\lvert#1\right\rvert}

\newcommand{\ol}[1]{\overline{#1}}

\newcommand{\wt}[1]{\widetilde{#1}}
\newcommand{\on}[1]{\operatorname{#1}}

\usepackage{extarrows}
\newcommand{\longhookrightarrow}{\ensuremath{\lhook\joinrel\relbar\joinrel\rightarrow}}

\DeclareMathOperator{\MI}{MI}
\DeclareMathOperator{\MV}{MV}
\DeclareMathOperator{\val}{val}
\DeclareMathOperator{\Spec}{Spec}
\DeclareMathOperator{\vol}{vol}

\DeclareMathOperator{\conv}{conv}
\DeclareMathOperator{\ord}{ord}

\renewcommand{\and}{\quad \text{ and } \quad}

\newcommand{\rad}{{\operatorname{rad}}}
\newcommand{\Res}{{\operatorname{Res}}}

\newcommand{\h}{{\operatorname{h}}}
\renewcommand{\c}{{\operatorname{c}}}

\newcommand{\an}{{\mathrm{an}}}
\newcommand{\can}{{\mathrm{can}}}
\newcommand{\supp}{{\operatorname{supp}}}
\newcommand{\graph}{{\operatorname{graph}}}

\DeclareMathOperator{\newton}{N}
\DeclareMathOperator{\Hom}{Hom}
\DeclareMathOperator{\e}{e}
\DeclareMathOperator{\K}{K}

\newcommand{\Xan}{X^{\mathrm{an}}}
\newcommand{\Yan}{Y^{\mathrm{an}}}
\newcommand{\Van}{V^{\mathrm{an}}}

\newcommand{\fM}{\mathfrak{M}}
\newcommand{\fN}{\mathfrak{N}}

\newcommand{\fp}{\mathfrak{p}}

\newcommand{\Gm}{\mathbb{G}_{\mathrm{m}}}
\newcommand{\GmK}{\mathbb{G}_{\mathrm{m},\mathbb{K}}}

\newcommand{\dd}{\, {\mathrm{d}}}

\newcommand{\m}{{\operatorname{m}}}

\let\div\relax
\DeclareMathOperator{\div}{div}

\renewcommand{\AA}{{\mathbb{A}}}
\newcommand{\CC}{{\mathbb{C}}}
\newcommand{\FF}{{\mathbb{F}}}
\newcommand{\GG}{{\mathbb{G}}}
\newcommand{\KK}{{\mathbb{K}}}

\newcommand{\PP}{{\mathbb{P}}}
\newcommand{\QQ}{{\mathbb{Q}}}
\newcommand{\RR}{{\mathbb{R}}}
\renewcommand{\SS}{{\mathbb{S}}}
\newcommand{\TT}{{\mathbb{T}}}
\newcommand{\ZZ}{{\mathbb{Z}}}

\newcommand{\MQ}{\mathfrak{M}_{\mathbb{Q}}}

\def\fM {{\mathfrak M}}

\def\cO {{\mathcal O}}

\newcommand{\bfm}{{\boldsymbol{m}}}

\newcommand{\bfq}{{\boldsymbol{q}}}

\newcommand{\bfu}{{\boldsymbol{u}}}

\newcommand{\bfx}{{\boldsymbol{x}}}

\newcommand{\bfz}{{\boldsymbol{z}}}

\newcommand{\bfalpha}{{\boldsymbol{\alpha}}}

\newcommand{\bflambda}{{\boldsymbol{\lambda}}}

\newcommand{\bfnu}{{\boldsymbol{\nu}}}

\newcounter{thm}
\numberwithin{equation}{section}
\numberwithin{thm}{section}

\theoremstyle{definition}
\newtheorem{definition}[thm]{Definition}

\newtheorem{remark}[thm]{Remark}
\newtheorem{example}[thm]{Example}

\theoremstyle{plain}
\newtheorem{lemma}[thm]{Lemma}
\newtheorem{proposition}[thm]{Proposition}
\newtheorem{theorem}[thm]{Theorem}
\newtheorem{corollary}[thm]{Corollary}

\newtheorem{question}[thm]{Question}
\newtheorem{prop-def}[thm]{Proposition-Definition}

\begin{document}

\title[An arithmetic Bern\v{s}tein-Ku\v{s}nirenko inequality]
{An arithmetic~Bern\v{s}tein-Ku\v{s}nirenko inequality}

\author[Mart\'inez]{C\'esar Mart\'inez}
\address{Laboratoire de math\'ematiques Nicolas Oresme, CNRS UMR 6139, Universit\'e de Caen. BP 5186, 14032 Caen Cedex, France}
\email{cesar.martinez@unicaen.fr}

\author[Sombra]{Mart{\'\i}n~Sombra}
\address{ICREA. Passeig Llu\'is Companys 23, 08010 Barcelona, Spain \vspace*{-2.5mm}}
\address{Departament de Matem\`atiques i
Inform\`atica, Universitat de Barcelona. Gran Via 585, 08007
Bar\-ce\-lo\-na, Spain}
\email{sombra@ub.edu}
\urladdr{\url{http://www.maia.ub.es/~sombra}}

\date{\today} \subjclass[2010]{Primary 14G40; Secondary 14C17, 14M25, 52A41}
\keywords{height of points, Laurent polynomials, 
  mixed integrals, toric varieties, $\bfu$-resultants} 
\thanks{Mart\'inez and Sombra were partially
  supported by the MINECO research projects MTM2012-38122-C03-02 and
  MTM2015-65361-P.  Mart\'inez was also partially supported by the
  CNRS project PICS 6381 ``G\'eom\'etrie diophantienne et calcul
  formel".}

\begin{abstract}
  We present an upper bound for the height of the isolated zeros in
  the torus of a system of Laurent polynomials over an adelic field
  satisfying the product formula. This upper bound is expressed in
  terms of the mixed integrals of the local roof functions associated
  to the chosen height function and to the system of Laurent
  polynomials. We also show that this bound is close to optimal in
  some families of examples.

  This result is an arithmetic analogue of the classical
  Bern\v{s}tein-Ku\v{s}nirenko theorem. Its proof is based on
  arithmetic intersection theory on toric varieties.
\end{abstract}

\vspace*{-8mm}

\maketitle

\vspace*{-10mm}

\section[Introduction]{Introduction} \label{sec:introduction}

The classical Bern\v{s}tein-Ku\v{s}nirenko theorem bounds the number
of isolated zeros of a system of Laurent polynomials over a field, in
terms of the mixed volume of their Newton polytopes \cite{kushnirenko:pnnm,Bernstein:nrsp}. This result,
initiated by~Ku\v{s}nirenko and put into final form by Bern\v stein,
is also known as the BKK theorem to acknowledge Khovanski\u\i's
contributions to this subject.  It shows how a geometric problem (the
counting of the number of solutions of a system of equations) can be
translated into a combinatorial, simpler one. It is commonly used to
predict when a given system of equations has a small number
of solutions. As such, it is a cornerstone of polynomial equation
solving and has motivated a large amount of work and results over the
past 25 years, see for instance
\cite{GelfandKapranovZelevinsky:drmd,Sturmfels:sspe,PhilipponSombra:rBKe}
and the references therein.

When dealing with Laurent polynomials over a field with an arithmetic
structure like the field of rationals, it might be also important to
control the arithmetic complexity or \emph{height} of their zero set.
In this paper, we present an arithmetic version of the BKK theorem,
bounding the height of the isolated zeros of a system of Laurent
polynomials over such a field. It is a refinement of the arithmetic
B\'ezout theorem that takes into account the finer monomial structure
of the system.

Previous results in this direction were obtained by Maillot
\cite{Maillot:GAdvt} and by the second author \cite{Sombra:msvtp}. Our
current result improves these previous upper bounds, and generalizes
them to adelic fields satisfying the product formula, and to height
functions associated to arbitrary nef toric metrized divisors.

Let $\KK$ be a field and $\ol \KK$ its algebraic closure.  Let
$M\simeq \ZZ^{n}$ be a lattice and set 
\begin{displaymath}
 \KK[M]\simeq \KK[x_{1}^{\pm1},\dots, x_{n}^{\pm1}] \and \TT_{M}=\Spec(\KK[M])\simeq \GmK^{n}   
\end{displaymath}
for its group $\KK$-algebra and algebraic torus over $\KK$,
respectively. For a family of Laurent polynomials
$f_{1},\dots, f_{n}\in \KK[M]$, we denote by $Z(f_{1},\dots, f_{n})$
the 0-cycle of $\TT_{M}$ given by the isolated solutions of the system
of equations
\begin{displaymath}
f_{1}=\dots= f_{n}=0
\end{displaymath}
with their corresponding multiplicities (Definition \ref{def:2}).

Set $M_{\RR}=M\otimes \RR\simeq \RR^{n}$. Let $\vol_{M}$ be the Haar
measure on $M_{\RR}$ normalized so that $M$ has covolume~$1$, and
let $\MV_{M}$ be the corresponding mixed volume function. For $i=1,\dots, n$,
let $\Delta_{i}\subset M_{\RR}$ be the Newton polytope of $f_{i}$. The
BKK theorem amounts to the upper bound
\begin{equation}
  \label{eq:1}
  \deg(Z(f_{1},\dots, f_{n}))\leq \MV_{M}(\Delta_{1},\dots, \Delta_{n}),
\end{equation}
which is an equality when the $f_{i}$'s are generic with respect to
their Newton polytopes \cite{kushnirenko:pnnm,Bernstein:nrsp}, see also
Theorem \ref{thm:BKK}.

Now suppose that $\KK$ is endowed with a set of places $\fM$, so that
the pair $(\KK,\fM)$ is an adelic field (Definition \ref{def:1}). 
Each place $v\in \fM$ consists of an absolute value $|\cdot|_{v}$ on
$\KK$ and a weight $n_{v}>0$. We assume that this set of places
satisfies the \emph{product formula}, namely, for all $\alpha\in
\KK^{\times}$, 
\begin{equation*}
  \sum_{v\in \fM}n_{v}\log|\alpha|_{v}=0.
\end{equation*}
The classical examples of adelic fields satisfying the product formula
are the global fields, that is, number fields and function fields of
regular projective curves. 

Let $X$ be toric compactification of $\TT_{M}$ and $\ol D_{0}$ a nef
toric metrized divisor on~$X$, see \S\ref{sec:metr-divis-heights} and
\S\ref{sec:toric-varieties} for details. This data gives a notion of
height for 0-cycles of~$X$, see \cite[Chapter 2]{BG:hdg} or
\S\ref{sec:metr-divis-heights}.  The height
\begin{displaymath}
  \h_{\ol D_{0}}(Z(f_{1},\dots, f_{n})) 
\end{displaymath}
is a nonnegative real number, and it is our aim to bound this
quantity in terms of the monomial expansion of the $f_{i}$'s.

The toric Cartier divisor $D_{0}$ defines a polytope
$\Delta_{0}\subset M_{\RR}$.  Following \cite{BPS14}, we associate to
$\ol D_{0}$ an adelic family of continuous concave functions
$\vartheta_{0,v}\colon \Delta_{0}\to \RR$, $v\in \fM$, called the
local roof functions of $\ol{D}_{0}$.  

For $i=1,\dots, n$, write
\begin{displaymath}
 f_{i}=\sum_{m\in M}\alpha_{i,m}\chi^{m} 
\end{displaymath}
with $\alpha_{i,m}\in \KK$.  Let
$N_{\RR}=M_{\RR}^{\vee}\simeq \RR^{n}$ be the dual space and, for each
place $v\in \fM$, consider the concave function
$\psi_{i,v}\colon N_{\RR}\to \RR$ defined by
\begin{equation}
  \label{eq:9}
  \psi_{i,v}(u)= 
  \begin{cases}
    -\log\Big( \displaystyle \sum_{m\in  M}|\alpha_{i,m}|_{v}
    \e^{-\langle m,u\rangle}\Big) & \text{ if } v \text{ is Archimedean},\\
    -\log\Big( \displaystyle \max_{m\in  M}|\alpha_{i,m}|_{v}
    \e^{-\langle m,u\rangle}\Big) & \text{ if } v \text{ is non-Archimedean}.
  \end{cases}
\end{equation}
The Legendre-Fenchel dual $\vartheta_{i,v}= \psi_{i,v}^{\vee}$ is a
continuous concave function on $\Delta_{i}$.

We denote by $\MI_{M}$ the mixed integral of a family of $n+1$ concave
functions on convex bodies of $M_{\RR}$ (Definition \ref{def:11}).  It
is the polarization of $(n+1)!$ times the integral of a concave
function on a convex body. It is a functional that is symmetric and
linear in each variable with respect to the sup-convolution of concave
functions, see \cite[\S8]{PhilipponSombra:rBKe} for details.

The following is the main result of this paper.

\begin{theorem}
  \label{thm:1}
  Let $f_1,\ldots,f_n\in\KK [M]$, and let $X$ be a proper toric
  variety with torus $\TT_{M}$ and $\ol D_0$ a nef toric metrized
  divisor on $X$. Let $\Delta_{0}\subset M_{\RR}$ be the polytope of
  $D_{0}$ and, for $v\in \fM$, let
  $\vartheta_{0,v}\colon \Delta_{0}\to \RR$ be $v$-adic roof function
  of $\ol D_{0}$. For $i=1,\dots, n$, let $\Delta_{i}\subset M_{\RR}$
  be the Newton polytope of $f_{i}$ and, for $v\in \fM$, let
  $\vartheta_{i,v}\colon \Delta_{i}\to \RR$ be the Legendre-Fenchel
  dual of the concave function in \eqref{eq:9}. Then
  \begin{equation}
    \label{eq:3}
\h_{\ol{D}_0}(Z(f_1,\ldots,f_n))
\leq \sum_{v\in\fM}n_v \MI(\vartheta_{0,v},\ldots,\vartheta_{n,v}).  
  \end{equation}
\begin{displaymath}
\end{displaymath}
\end{theorem}

Using the basic properties of the mixed integral, we can bound the
terms in the right-hand side of \eqref{eq:3} in terms of mixed
volumes. From this, we can derive the bound (Corollary \ref{cor:height2})
\begin{multline*}
\h_{\ol{D}_0}(Z(f_1,\ldots,f_n))
\leq
\MV_{M}(\Delta_1,\ldots,\Delta_n) \Big(\sum_{v\in\fM}\max\vartheta_{0,v}\Big)\\[-2mm]
+ 
\sum_{i=1}^n\MV_{M}(\Delta_0,\ldots,{\Delta}_{i-1},
\Delta_{i+1},\ldots,\Delta_n) \ell(f_i), 
\end{multline*}
where   $\ell(f_{i})$  denotes  the (logarithmic) length of
$f_{i}$ (Definition~\ref{def:l1norm}). This bound  might be compared
with the one  given by the arithmetic B\'ezout theorem (Corollary~\ref{cor:3}).

The following example illustrates a typical application of these
results. It concerns two height functions applied to the same
0-cycle. Our upper bounds are close to optimal for both of them and,
in particular, they reflect their very different behavior on this
family of Laurent polynomials.

\begin{example}
  \label{exm:1}
  Take two integers $d,\alpha \ge 1$ and consider the system of Laurent
  polynomials
  \begin{displaymath}
  f_{1}=x_{1}-\alpha , \quad 
f_{2}=x_{2}-\alpha x_{1}^{d} \quad, \dots, \quad
  f_{n}=x_{n}-\alpha x^d_{n-1} \quad\in \QQ[x_{1}^{\pm1},\dots, x_{n}^{\pm1}].
  \end{displaymath}
  The 0-cycle $Y:=Z(f_{1},\dots, f_{n})$ of $\GG^{n}_{\mathrm{m},\QQ}$
  is the single point
  $(\alpha , \alpha ^{d+1}, \dots, \alpha ^{d^{n-1}+\dots+d+1})$ with
  multiplicity 1.

Let $\PP_{\QQ}^{n}$ be the $n$-dimensional projective space over $\QQ$
and $\ol{E}^{\can}$ the divisor of the hyperplane at infinity,
equipped with the canonical metric. Its associated height function is
the Weil height.  We consider two toric compactifications
$X_{1}$ and $X_{2}$ of
$\Gm^{n}$. These are given by compactifying the torus via the
equivariant  embeddings
$\iota_{i}\colon \Gm^{n}\hookrightarrow \PP^{n}_{\QQ}$, $i=1,2$,
respectively defined, for
$p=(p_{1},\dots,p_{n})\in \Gm^{n}(\ol\QQ)=(\ol\QQ^{\times})^{n}$, by
\begin{displaymath}
  \iota_{1}(p)=(1:p_{1}:\dots:p_{n}) \and   \iota_{2}(p)=(1:p_{1}:p_{2}p_{1}^{-d}:\dots:p_{n}p_{n-1}^{-d}).
\end{displaymath}
Set  $\ol{D}_{i}=\iota_{i}^{*}\, \ol{E}^{\can}$, $i=1,2$,
which are nef toric metrized divisors on $X_{i}$.  By an
explicit computation, we show that
\begin{equation*}
  \h_{\ol{D}_{1}}(Y) =
  \Big(\sum_{i=1}^{n}d^{i-1}\Big)\log(\alpha ) \and   
  \h_{\ol{D}_{2}}(Y) = \log(\alpha ). 
\end{equation*}
On the other hand, the
upper bounds given by Theorem \ref{thm:1} are
\begin{displaymath}
  \h_{\ol{D}_{1}}(Y)  \le 
  \Big(\sum_{i=1}^{n}d^{i-1}\Big)\log(\alpha +1) \and
  \h_{\ol{D}_{2}}(Y) \le  n \log(\alpha +1),
\end{displaymath}
see Example \ref{exm:2} for details. 
\end{example}

To the best of our knowledge, the first arithmetic analogue of the BKK
theorem was proposed by Maillot \cite[Corollaire
8.2.3]{Maillot:GAdvt}, who considered the case of canonical toric
metrics.  Another result in this direction was obtained by the second
author for the unmixed case and also canonical toric metrics
\cite[Th\'eor\`eme 0.3]{Sombra:msvtp}. Theorem~\ref{thm:1} improves
and refines these previous upper bounds, and generalizes them to
adelic fields satisfying the product formula and to height functions
associated to arbitrary nef toric metrized divisors, see
\S\ref{sec:examples} for details.

The key point in the proof of Theorem \ref{thm:1} consists of the
construction, for each Laurent polynomial $f_{i}$, of a nef toric
metrized divisor $\ol D_{i}$ on a proper toric variety $X$, such
that~$f_{i}$ corresponds to a small section of $\ol D_{i}$
(Proposition \ref{prop:10} and Lemma \ref{lemm:3}). The proof then
proceeds by applying the constructions and results of \cite{BPS14,
  BMPS12} and basic results from arithmetic intersection theory.

Trying to keep our results at a similar level of generality as those
in \cite{BPS14}, we faced difficulties to define and study global
heights of cycles over adelic fields. This lead us to a more detailed
study of these notions. In particular, we give a new notion of adelic
field extension that preserves the product formula (Proposition
\ref{prop:finiteextension}) and a well-defined notion of global height
for cycles with respect to metrized divisors that are generated by
small sections (Proposition-Definition \ref{prop:4}).

As an application of Theorem \ref{thm:1}, we give an upper bound for
the size of the coefficients of the $\bfu$-resultant of the direct
image under a monomial map of the solution set of a system of Laurent
polynomial equations. 

For the simplicity of the exposition, set $\KK=\QQ$ and
$M=\ZZ^{n}$. Let $r\ge 0$,
$\bfm_{0}=(m_{0,0},\dots, m_{0,r})\in (\ZZ^{n})^{r+1}$ and
$\bfalpha_{0}=(\alpha_{0,0},\dots, \alpha_{0,r})\in (\ZZ\setminus
\{0\})^{r+1}$,
and consider the map
$\varphi_{\bfm_{0},\bfalpha_{0}}\colon \GG_{\mathrm{m},\QQ}^{n}\to \PP_{\QQ}^{r}$ defined
by
\begin{equation}
  \label{eq:7}
\varphi_{\bfm_{0},\bfalpha_{0}}(p)=
(\alpha_{0,0}\chi^{m_{0,0}}(p): \dots :
\alpha_{0,r}\chi^{m_{0,r}}(p) ).
\end{equation}

For a 0-cycle $W$ of $\PP_{\QQ}^{r}$, let $\bfu=(u_0,\ldots,u_r)$ be a
set of $r+1$ variables and denote by
$\Res(W) \in \ZZ[u_{1},\dots, u_{r}]$ its primitive $\bfu$-resultant
(Definition \ref{def:13}). It is well-defined up a sign. For a vector
$\bfalpha$ with integer entries, we denote by $\ell(\bfalpha)$ the
logarithm of the sum of the absolute values of its entries.

\begin{theorem} \label{thm:2} Let
  $f_{1},\dots, f_{n}\in \ZZ[x_{1}^{\pm1},\dots,x_{n}^{\pm1}]$,
  $\bfm_{0}\in (\ZZ^{n})^{r+1}$ and
  $\bfalpha_{0}\in (\ZZ\setminus \{0\})^{r+1}$ with $r\ge 0$.  Set
  $\Delta_{0}=\conv(m_{0,0},\dots, m_{0,r})\subset \RR^{n}$ and let
  $\varphi$ be the monomial map associated to $\bfm_{0}$ and
  $\bfalpha_{0}$ as in~\eqref{eq:7}. For $i=1,\dots, n$, let
  $\Delta_{i}\subset {\RR}^{n}$ be the Newton polytope of $f_{i}$, and
  $\bfalpha_{i}$ the vector of nonzero coefficients of $f_{i}$.  Then
\[
\ell(\Res(\varphi_*Z(f_1,\ldots,f_n)))
\leq
\sum_{i=0}^n
\MV_M(\Delta_0,\ldots,\Delta_{i-1},\Delta_{i+1},\ldots,\Delta_n) \, \ell (\bfalpha_{i}).
\]
\end{theorem}

The paper is organized as follows. In \S\ref{sec:intersection-theory}
we recall some preliminary material on intersection theory and on the
algebraic geometry of toric varieties.  In
\S\ref{sec:adel-fields-extens} we study adelic fields satisfying the
product formula.  In \S\ref{sec:metr-divis-heights} we recall the
notions of metrized divisors and its associated measures and heights,
with an emphasis on the 0-dimensional case.  In
\S\ref{sec:toric-varieties} we explain the notation and basic
constructions of the arithmetic geometry of toric varieties.  In
\S\ref{sec:theorems} we prove Theorem \ref{thm:1}, whereas in
\S\ref{sec:examples} we give examples illustrating the applications of
our bounds, and prove Theorem \ref{thm:2}.

\medskip \noindent {\textbf{Acknowledgments.}} We thank Jos\'e Ignacio
Burgos, Roberto Gualdi and Patrice Philippon for useful discussions.
We also thank Walter Gubler, Philipp Habegger and the referee for
their helpful comments and suggestions for improvement on a previous
version of this paper. Part of this work was done while the authors met
at the Universitat de Barcelona and at the Universit\'e de Caen. We
thank both institutions for their hospitality.

\section{Intersection theory and toric varieties} \label{sec:intersection-theory}

In this section, we recall the proof of the
Bern\v{s}tein-Ku\v{s}nirenko theorem using intersection theory on
toric varieties, which is the model that we follow in our treatment of
the arithmetic version of this result.  Presenting this proof also
allows us to introduce the basic definitions and results on the
intersection of Cartier divisors with cycles, and on the algebraic
geometry of toric varieties. For more details on these subjects, we
refer to~\cite{Ful98, Lazarsfeld:posit_I} and to \cite{Ful93}.

Let $K$ be an infinite field and $X$ a variety over~$K$ of dimension
$n$.  For $0\le k\le n$, the \emph{group of $k$-cycles}, denoted by
$Z_k(X)$, is the free abelian group on the $k$-dimensional irreducible
subvarieties of $X$.  Thus, a \emph{$k$-cycle} is a finite formal sum
\[
Y=\sum_V m_V V
\]
where the $V$'s are $k$-dimensional irreducible subvarieties of $X$
and the $m_V$'s are integers.  The \emph{support} of $Y$, denoted by
$|Y|$, is the union of the subvarieties $V$ such that $m_{V}\ne
0$.
The cycle $Y$ is \emph{effective} if $m_V\geq 0$ for every $V$.  Given
$Y, Y'\in Z_{k}(X)$, we write $Y'\leq Y$ whenever $Y-Y'$ is effective.

Let $Z$ be a subscheme of $X$ of pure dimension $k$.  For an
irreducible component $V$ of $Z$, we denote by $\cO_{V,Z}$ the local ring of
$Z$ along $V$, and by $l_{\cO_{V,Z}}(\cO_{V,Z})$ its length as an
$\cO_{V,Z}$-module.  The {$k$-cycle} associated to $Z$ is then
defined as
\[
[Z]=\sum l_{\cO_{V,Z}}(\cO_{V,Z})\, V\mbox{,}
\]
the sum being over the irreducible components of $Z$. 

Let $V$ be an irreducible subvariety of $X$ of codimension one and $f$
a regular function on an open subset $U$ of $X$ such that $U\cap
V\neq\emptyset$. The \emph{order of vanishing} of $f$ along $V$ is
defined as
\[
\ord_{V}(f)=l_{\mathcal{O}_{V,X}(U)}(\mathcal{O}_{V,X}(U)/(f))\mbox{.}
\]
For a Cartier divisor $D$ on $X$, the \emph{order of vanishing} of $D$
along $V$ is  defined as
\[
\ord_{V}(D)=\ord_{V}(g)-\ord_{V}(h),
\]
with $g,h\in\mathcal{O}_{V,X}(U)$ such that $g/h$ is a local equation
of $D$ on an open subset $U$ of $X$ with $U\cap V\neq \emptyset$.  This
definition does not depend on the choice of $U$, $g$ and~$h$.
Moreover, $\ord_V(D)= 0$ for all but a finite number of $V$'s. The
\emph{support} of $D$, denoted by $\abs{D}$, is the union
of these subvarieties $V$ such that  $\ord_V(D)\ne 0$. The Weil divisor associated to $D$ is then
defined as
\begin{equation}\label{eq:intersectioncycle}
D\cdot X=\sum_V \ord_V(D)\, V\mbox{,}
\end{equation}
 the sum being over the irreducible components of~$\abs{D}$. 

 Now let $W$ be an irreducible subvariety of $X$ of dimension~$k$.  If
 $W\not\subset\abs{D}$, then~$D$ restricts to a Cartier divisor on
 $W$.  In this case, we define $D\cdot W$ as the Weil divisor of~$W$
 obtained by restricting (\ref{eq:intersectioncycle}) to~$W$.  This
 gives a $(k-1)$-cycle of~$X$.  If $W\subset\abs{D}$, then we set
 $D\cdot W=0$, the zero element of $Z_{k-1}(X)$. We extend by linearity
 this intersection product to a morphism
\begin{equation*}
  Z_{k}(X)\longrightarrow Z_{k-1}(X), \quad Y\longmapsto D\cdot Y,
\end{equation*}
with the convention that $Z_{-1}(X)=0$, the zero group.

For $0\le r\leq n$ and Cartier divisors $D_i$ on~$X$, $i=1,\dots, r$,
we define inductively the intersection product
$\prod_{i=1}^{r}D_i \in Z_{n-r}(X)$ by
\begin{equation*}
    \prod_{i=1}^{t}D_i= 
    \begin{cases}
      X & \text{ if } t=0, \\
D_{1}\cdot \prod_{i=2}^{t} D_{i} & \text{ if } 1\le t\le r.
    \end{cases}
\end{equation*}

\begin{definition} 
\label{def:8}
Let $Y$ be a $k$-cycle of $X$ and $D_1,\ldots,D_r$ 
Cartier divisors on~$X$, with~$r\leq k$.  We say that $D_1,\ldots,D_r$
\emph{intersect} $Y$ \emph{properly} if, for every subset
$I\subset\lbrace 1,\ldots,r\rbrace$,
\[
\dim \Big( |Y|\cap\bigcap_{i\in I}\abs{D_i}\Big)=k-\#I\mbox{.}
\]
\end{definition}

If $D_1,\ldots,D_r$ intersect $X$ properly, then the 
cycle $\prod_{i=1}^{r}D_{i}$ does not depend on the order of the
$D_i$'s~\cite[Corollary~2.4.2]{Ful98}.  This conclusion does not
necessarily hold if these divisors do not intersect properly.

\begin{example}
  Let $X=\AA^2_K$ and consider the principal Cartier divisors
  $D_1=\div(x_{1}x_{2})$ and $D_2=\div(x_{1})$. Then
\begin{displaymath}
	D_1\cdot D_2=0 \and D_2\cdot D_1= (0,0)\mbox{.}
\end{displaymath}

\end{example}

\begin{proposition}\label{prop:intersection}
  Let $X$ be a Cohen-Macaulay variety over $K$ of pure dimension~$n$
  and $D_1,\ldots,D_n$ Cartier divisors on $X$. Let $s_{i}$ be a
  global section of $\cO(D_{i})$, $i=1,\dots, n$, and write
\begin{equation}\label{eq:27}
  \prod_{i=1}^{n}\div(s_{i})=\sum_{p}m_{p}\, p \in Z_{0}(X),
\end{equation}
where the sum is over the closed points $p$ of $X$ and $m_{p}\in
\ZZ$. This 0-cycle is effective and, for each isolated closed point
$p$ of the intersection $\bigcap_{i=1}^{n}\abs{\div(s_{i})}$, 
\[
m_p=\dim_{K} (\mathcal{O}_{p,X}(U)/(f_1,\ldots,f_n)),
\]
where $U$ is a trivializing neighborhood of $p$, and $f_i$ is a
defining function for $s_i$ on $U$, $i=1,\dots, n$.
\end{proposition}

\begin{proof}
  The fact that the cycle in \eqref{eq:27} is effective follows from
  the hypothesis that the~$s_{i}$'s are global sections.

  For the second statement, by possibly replacing $U$ with a smaller
  open neighborhood of $p$, we can assume that
  $\div(s_{1}), \dots, \div(s_{n})$ intersect $X$ properly on~$U$, and
  so this intersection is of dimension $0$. By \cite[Proposition~{7.1}
  and Example~{7.1.10}]{Ful98},
\[
m_p=l_{\mathcal{O}_{p,X}(U)}(\mathcal{O}_{p,X}(U)/(f_1,\ldots,f_n))\mbox{.}
\]
By \cite[Example A.1.1]{Ful98},  we have the equality
\[
l_{\mathcal{O}_{p,X}(U)}(\mathcal{O}_{p,X}(U)/(f_1,\ldots,f_n))=
\dim_{K} (\mathcal{O}_{p,X}(U)/(f_1,\ldots,f_n))\mbox{,}
\]
completing the proof.
\end{proof}

For the rest of this section, we assume that the variety $X$ is
projective. With this hypothesis, Chow's moving lemma allows to
construct, given a cycle and a family of Cartier divisors, another
family of linearly equivalent Cartier divisors intersecting the given cycle properly,
in the sense of Definition \ref{def:8}. 

\begin{definition}
  \label{def:7}
  Let $Y$ be a $k$-cycle of $X$ and $D_{1},\dots, D_{k}$ 
  Cartier divisors on $X$. The \emph{degree} of $Y$ with respect to
  $D_{1},\dots, D_{k}$, denoted by $\deg_{D_{1},\dots, D_{k}}(Y)$, is
   inductively defined by the rules:
\begin{enumerate}
\item \label{item:5} if $k=0$, write $Y=\sum_{p}m_{p} \, p $ and set 
$\deg(Y)=\sum_{p}m_{p}\, [\K(p):K]$;
\item \label{item:6} if $k\ge1$, choose a rational section $s_{k}$ of
  $\cO(D_{k})$ such that $\div(s_{k})$ intersects $Y$ properly, and
  set $\deg_{D_{1},\dots, D_{k}}(Y)= \deg_{D_{1},\dots,
    D_{k-1}}(\div(s_{k})\cdot Y)$.
\end{enumerate}
\end{definition}

The degree of a cycle with respect to a family of Cartier divisors does not
depend on the choice of the rational section $s_{k}$ in \eqref{item:6}, see for instance
\cite[\S 2.5]{Ful98} or \cite[\S 1.1.C]{Lazarsfeld:posit_I}. 

A Cartier divisor $D$ on $X$ is \emph{nef} if $\deg_{D}(C)\ge 0$ for
every irreducible curve $C$ of $X$. By Kleiman's theorem \cite[\S
1.4.B]{Lazarsfeld:posit_I}, for a family of nef Cartier divisors
$D_{1},\dots, D_{k}$ on~$X$ and an effective $k$-cycle $Y$ of $X$, 
\begin{equation}
  \label{eq:19}
  \deg_{D_{1},\dots, D_{k}}(Y)\ge 0.
\end{equation}

\begin{proposition}
  \label{prop:5}
  Let $Y$ be an effective $k$-cycle of $X$ and $D_{1},\dots, D_{k}$
  nef Cartier divisors on $X$. Let $s_{k}$ be a global section of
  $\cO(D_{k})$. Then
  \begin{displaymath}
    0\le \deg_{D_{1},\dots, D_{k-1}}( \div(s_{k})\cdot Y)     \le \deg_{D_{1},\dots, D_{k}}(Y).
  \end{displaymath}
\end{proposition}

\begin{proof}
  Since $Y$ is effective and $s_{k}$ is a  global section, 
  $ \div(s_{k})\cdot Y$ is also effective. Since $D_{1},\dots,
  D_{k-1}$ are nef, by \eqref{eq:19} we have that $\deg_{D_{1},\dots,
    D_{k-1}}( \div(s_{k})\cdot Y)\ge 0$, proving the lower bound.

  For the upper bound, we reduce without loss of generality to the
  case when $Y=V$ is an irreducible subvariety of dimension $k$. If
  $V\subset |\div(s_{k})|$, then $\div(s_{k})\cdot Y=0\in
  Z_{k-1}(X)$. Hence $\deg(\div(s_{k})\cdot Y)=0$ and the bound
  follows from the nefness of the~$D_{i}$'s. Otherwise, from the
  definition of the degree,
  \begin{displaymath}
    \deg_{D_{1},\dots, D_{k-1}}( \div(s_{k})\cdot V) =
  \deg_{D_{1},\dots, D_{k}}(V),
  \end{displaymath}
which completes the proof. 
\end{proof}

\begin{corollary}
  \label{cor:1}
  Let $D_1,\ldots,D_n$ be nef Cartier divisors on $X$ and, for
  $i=1,\dots, n$, let $s_{i}$ be a global section of
  $\cO(D_{i})$. Then
\begin{displaymath}
 0\leq \deg\bigg(\prod_{i=1}^{n}\div(s_{i}) \bigg)\le \deg_{D_{1},\dots, D_{n}}(X).
\end{displaymath}
\end{corollary}

We now turn to toric varieties. Let $M\simeq \ZZ^{n}$ be a lattice and
set
\begin{equation}\label{eq:29}
  K[M]\simeq K[x_{1}^{\pm1},\dots, x_{n}^{\pm1}] \and \TT=\Spec(K[M])\simeq \GG_{\text{m},K}^n
\end{equation}
for its group $K$-algebra and algebraic torus over $K$, respectively.
The elements of $M$ correspond to the characters of $\TT$ and, given
$m\in M$, we denote by $\chi^{m}\in \Hom(\TT, \GG_{\text{m},K})$ the
corresponding character. Set also $M_{\mathbb{R}}=M\otimes\mathbb{R}$.

Let $N=M^{\vee} \simeq\mathbb{Z}^n$ be the dual lattice and set 
$N_{\mathbb{R}}=N\otimes\mathbb{R}$.  Given a complete fan~$\Sigma$ in
$N_\RR$, we denote by $X_\Sigma$ the associated toric variety with
torus~$\TT$.  It is a proper normal variety over $K$ containing $\TT$
as a dense open subset. When the fan $\Sigma$ is \emph{regular}, in the sense
that it is induced by a piecewise linear concave function on $N_{\RR}$,
the toric variety $X_{\Sigma}$ is projective.

Set $X=X_{\Sigma}$ for short. Let $D$ be a toric Cartier divisor on
$X$, and denote by $\Psi_D$ its associated virtual support function on
$\Sigma$. This is a piecewise linear function
$\Psi_{D}\colon N_{\RR}\to \RR$ satisfying that, for each cone
$\sigma\in \Sigma$, there exists $m\in M$ such that, for all $u\in
\sigma$, 
\begin{displaymath}
\Psi_{D}(u)=\langle m, u\rangle.  
\end{displaymath}
The condition that $\Psi_D$ is concave is both equivalent to the
conditions that $D$ is nef and that the line bundle $\cO(D)$ is
globally generated.  This line bundle $\mathcal{O}(D)$ is a subsheaf
of the sheaf of rational functions of $X$.  For each $m\in M$, the
character $\chi^{m}$ is a rational function of $X$, and so it induces
a rational section of $\cO(D)$ that is regular and nowhere vanishing
on $\TT$.  The rational section corresponding to the point $m=0$ is
called the \emph{distinguished rational section} of $\cO(D)$ and
denoted by $s_{D}$.

The toric Cartier divisor $D$ also determines the lattice polytope of
$M_{\RR}$ given~by
\begin{equation*}
\Delta_D=\lbrace x\in M_{\RR}\mid\; \langle  x, u\rangle \geq \Psi_D( u)\mbox{ for every } u\in N_{\RR}\rbrace\mbox{.}  
\end{equation*}
A  rational section corresponding to a  point $m\in M$ is global if and
only if $m\in\Delta_{D}$. The global sections corresponding to the
lattice points of $\Delta_{D}$ form  a
$K$-basis for the space of global sections of $\cO(D)$.
Identifying each character $\chi^{m}$ with the corresponding rational
section $\varsigma_m$ of~$\cO(D)$, we have the decomposition
\begin{equation}
  \label{eq:5}
\Gamma(X,\mathcal{O}(D))=\displaystyle\bigoplus_{
  m\in\Delta_{D}\cap M}K \cdot \varsigma_{m}\mbox{.}
\end{equation}

Now let $\Delta_{1},\dots, \Delta_{r}$ be lattice polytopes in
$M_{\RR}$. For each $\Delta_{i}$, we consider its \emph{support function},
which is the piecewise linear concave function with lattice slopes
$\Psi_{\Delta_{i}} \colon N_{\RR}\to \RR$ given by
\begin{equation}
  \label{eq:12}
\Psi_{\Delta_{i}} (u)= \min_{ x\in \Delta_{i}}\langle  x, u\rangle\mbox{.}  
\end{equation}

Let $\Sigma$ be a regular complete fan in $N_{\RR}$ compatible with
the collection $\Delta_{1},\dots, \Delta_{r}$, in the sense that the
$\Psi_{\Delta_{i}}$'s are virtual support functions on $\Sigma$. Such
a fan can be constructed by taking any regular complete fan in
$N_{\RR}$ refining the complex of cones that are normal to the faces
of $\Delta_{i}$, for all $i$. Let $X$ be the toric variety
corresponding to this fan and $D_{i}$ the toric Cartier divisor on $X$
corresponding to these virtual support functions.  By construction,
$\Psi_{\Delta_i}$ is concave. Hence~$D_i$ is nef and $\cO(D_i)$ is
globally generated, and its associated polytope coincides
with~$\Delta_{i}$.

Let $\vol_M$ be the Haar measure on $M_\RR$ such that $M$ has
covolume~$1$, and take $r=n$. The \emph{mixed volume} of $\Delta_1,\ldots,\Delta_n$
is defined as the alternating sum
\begin{equation*}
\MV_M(\Delta_1,\ldots,\Delta_n)=\sum_{j=1}^n(-1)^{n-j}\hspace{-3mm}\sum_{1\leq i_1 <\cdots <i_j\leq n}
\vol_M(\Delta_{i_1}+\cdots +\Delta_{i_j}).
\end{equation*}
A fundamental result in toric geometry states that the
degree of a toric variety with respect to a family of  nef toric
Cartier divisors is given by  the mixed volume of its
 polytopes \cite[\S5.4]{Ful93}. In our present setting,
this amounts to the formula
\begin{equation}\label{eq:degMV}
\deg_{D_1,\ldots,D_n}(X)=\MV_M(\Delta_1,\ldots,\Delta_n).
\end{equation}

We turn to 0-cycles of the torus defined by families of Laurent polynomials.

\begin{definition}
  \label{def:2}
  Let $f_1,\ldots,f_n\in K[M]$ and denote by $V(f_1,\ldots,f_n)_0$ the
  set of isolated closed points in the variety defined by this family
  of Laurent polynomials. For each $p\in V(f_{1},\dots, f_{n})_{0}$,
  let $\mathfrak{m}_p$ be the maximal ideal of $K[M]$ corresponding
  to~$p$ and set
\[
\mu_p=\dim_K(K[M]_{\mathfrak{m}_p}/(f_1,\ldots,f_n)).
\]
The
\emph{0-cycle associated to $f_{1},\dots, f_{n}$} is defined as
\[
Z(f_1,\ldots,f_n)=\sum_{p\in V(f_1,\ldots,f_n)_0}\mu_p \, p\in Z_0(\TT).
\]
\end{definition}

Let $f=\sum_{m\in M}\alpha_{m}\chi^{m}\in K[M]$ be a Laurent
polynomial. Its \emph{support} is defined as the finite subset of $M$
of the exponents of its nonzero terms, that is
$\supp(f)=\{m\mid \alpha_{m}\ne 0\}$. The \emph{Newton polytope} of
$f$ is the lattice polytope in $M_{\RR}$ given by the convex hull of
its support, that is $\newton(f)=\conv(\supp(f))$.

\begin{proposition}\label{prop:0cycles}
  Let $f_1,\ldots,f_n\in K[M]$.  Let~$\Sigma$ be a regular complete
  fan in $N_{\RR}$ compatible with the Newton polytopes of the
  $f_{i}$'s and, for $i=1,\dots, n$, let $D_{i}$ be the Cartier
  divisor on $X_{\Sigma}$ associated to $\newton(f_i)$ and $s_{i}$ the
  global section of $\cO(D_{i})$ corresponding to~$f_{i}$ as in
  \eqref{eq:5}.  Write $\prod_{i=1}^{n}\div(s_i)=\sum \nu_p \, p$,
  where the sum is over the closed points of $X_{\Sigma}$. Then
  \begin{enumerate}
  \item \label{item:1} for every $p\in V(f_1,\ldots,f_n)_0$, we have
    $\nu_p=\dim_K(K[M]_{\mathfrak{m}_p}/(f_1,\ldots,f_n))$;
\item \label{item:2} the inequality $Z(f_1,\ldots,f_n)\leq
  \prod_{i=1}^n\div(s_i)$ holds.
  \end{enumerate}
\end{proposition}

\begin{proof}
  We have that
  $\bigcap_{i=1}^{n}|\div(s_{i})| = V(f_{1},\dots, f_{n})$.  Since
  $\TT$ is Cohen-Macaulay, Proposition \ref{prop:intersection} gives
  the first statement.  Since the sections $s_{i}$ are global, the
  0-cycle $ \prod_{i=1}^n\div(s_i)$ is effective. Hence, the second
  statement follows directly from the first one.
\end{proof}

Finally, we prove the version of the Bern\v{s}tein-Ku\v{s}nirenko
theorem in \eqref{eq:1}.

\begin{theorem}\label{thm:BKK}
  Let $f_1,\ldots,f_n\in K[M]$.  Then
\[
\deg(Z(f_1,\ldots,f_n))\leq \MV_M(\newton(f_1),\ldots,\newton(f_n)).
\]
\end{theorem}

\begin{proof}
  This follows  from
  Proposition~\ref{prop:0cycles}\eqref{item:2}, Corollary~\ref{cor:1}
  and the formula~\eqref{eq:degMV}.
\end{proof}

\section{Adelic fields and finite extensions}\label{sec:adel-fields-extens}

In this section, we consider adelic fields following \cite{BPS14}. We
also give a new notion of adelic field extension that behaves better
than the one in \emph{loc. cit.}.  With this definition, the product
formula is preserved when passing to finite extensions.

\begin{definition}
  \label{def:1}
  Let $\mathbb{K}$ be an infinite field and $\mathfrak{M}$ a set of~places.
  Each place $v\in\mathfrak{M}$ is a pair consisting of an absolute
  value $\abs{\cdot}_v$ and a positive real weight~$n_v$.  We say that
  $(\mathbb{K},\mathfrak{M})$ is an \emph{adelic field} if
\begin{enumerate}
\item for each $v\in\mathfrak{M}$, the absolute value $\abs{ \cdot}_v$ is either Archimedean or associated to a nontrivial discrete valuation;
\item for each $\alpha\in\mathbb{K}^\times$, we have that $\abs{\alpha}_v=1$ for all but a finite number of~$v\in\mathfrak{M}$.
\end{enumerate}
An adelic field $(\mathbb{K},\mathfrak{M})$  satisfies the
\emph{product formula} if, for every $\alpha\in\mathbb{K}^\times$, 
\[
\displaystyle\prod_{v\in\mathfrak{M}}\abs{\alpha}_v^{n_v}=1\mbox{.}
\]
\end{definition}

Let $(\KK, \fM)$ be an adelic field. For each place
$v\in\mathfrak{M}$, we denote by $\KK_v$ the {completion} of $\KK$
with respect to the absolute value~$\abs{\cdot}_v$. By a theorem of
Ostrowski, if $v$ is Archimedean, then $\KK_{v}$ is isomorphic to
either $\RR$ or $\CC$ \cite[Chapter 3, Theorem~1.1]{Cassels:lf}. In
particular, an adelic field has only a finite number of Archimedean
places.

\begin{example}\label{ex:numberfield}
  Let $\MQ$ be the set of places of $\QQ$ consisting of the
  Archimedean and $p$-adic absolute values of $\QQ$, normalized in the
  standard way, and with all the weights equal to~$1$. The adelic
  field $(\QQ,\MQ)$ satisfies the product formula.
\end{example}

\begin{example}\label{ex:funtionfield}
  Let $\K(C)$ denote the function field of a regular projective curve
  $C$ over a field $\kappa$.  To each closed point $v\in C$ we associate
  the absolute value and weight given, for $f\in \K(C)^{\times}$, by
  \begin{equation}
    \label{eq:11}
\abs{f}_v=c_\kappa^{-\ord_v(f)} \and 
n_v=[\K(v):\kappa],
  \end{equation}
where $\ord_v(f)$ denotes the order of vanishing of $f$ at~$v$ and
\begin{equation}
  \label{eq:23}
c_\kappa=
\left\lbrace
\begin{array}{ll}
\e&\mbox{if } \# \kappa=\infty\mbox{,}\\
\# \kappa&\mbox{if } \# \kappa<\infty\mbox{.}
\end{array}
\right.
\end{equation}
The set of places $\mathfrak{M}_{\K(C)}$ is indexed by the closed
  points of $C$, and consists of these absolute values and weights.
  The pair $(\K(C),\mathfrak{M}_{\K(C)})$ is an adelic field which
  satisfies the product formula.
\end{example}

\begin{lemma}\label{prop:artiniandecomposition}
  Let $\FF$ be a
  finite extension of~$\KK$ and $v\in \fM$.  Then
  \begin{equation}
    \label{eq:20}
\FF\otimes_{\KK}\KK_v\simeq\bigoplus_{w}E_w\mbox{,}    
  \end{equation}
  where the sum is over the absolute values $|\cdot|_{w}$ on $\FF$
  whose restriction to $\KK_{v}$ coincides with~$|\cdot|_{v}$, and
  where the $E_w$'s are local Artinian $\KK_v$-algebras with maximal
  ideal~$\fp_w$. For each $w$, we have $E_w/\fp_w\simeq\FF_w$.
\end{lemma}

\begin{proof}
Since $\KK\hookrightarrow \FF$ is a finite extension, the tensor
product $\FF\otimes \KK_{v}$ is an Artinian $\KK_v$-algebra. 
By  the structure theorem for Artinian algebras, 
\[
\FF\otimes_\KK\KK_v\simeq\bigoplus_{i\in I}E_i\mbox{,}
\]
where $I$ is a finite set and the $E_i$'s are local Artinian
$\KK_v$-algebras.  Let $\fp_i$ be the maximal ideal of~$E_i$, for
each~$i$. These are the only prime ideals of $\FF\otimes\KK_{v}$, and
so $\rad(\FF\otimes\KK_v)=\bigcap_{i\in I}\fp_i$.

Each $w$ in the decomposition \eqref{eq:20} corresponds to an absolute value $|\cdot|_{w}$ on $\FF$
extending $|\cdot|_{v}$, and there is a natural inclusion
$\FF\hookrightarrow \FF_{w}$. The diagonal morphism
$\FF\rightarrow\bigoplus_{w}\FF_w$ extends to a map of $\KK_{v}$-vector
spaces
\[
\FF\otimes_\KK\KK_v\longrightarrow\bigoplus_{w}\FF_w\mbox{.}
\]
By \cite[Chapitre VI, §8.2 Proposition~11(b)]{Bou64}, this morphism is
surjective and its kernel is the radical ideal of $\FF\otimes\KK_v$.
Therefore
\begin{equation}
  \label{eq:10}
\bigoplus_{i\in I}E_i/\fp_i=\bigg(\bigoplus_{i\in
  I}E_i\bigg)\bigg\slash \rad(\FF\otimes\KK_v)\simeq \bigoplus_{w}\FF_w\mbox{.}
\end{equation}
The summands in both extremes of \eqref{eq:10} are fields over
$\KK_{v}$, and so local Artinian $\KK_{v}$-algebras.  By the
uniqueness of the decomposition in the structure theorem for Artinian
algebras, there is a bijection between the elements in $I$ and the
$w$'s, identifying each $i\in I$ with the unique $w$ such that
$E_i/\fp_i\simeq\FF_w$.
\end{proof}

The following definition was introduced by Gubler in the context of
$M$-fields, see \cite[Remark 2.5]{Gubler:hsmf}.

\begin{definition}
\label{def:3}
Let $(\KK,\mathfrak{M})$ be an adelic field and $\FF$ a finite
extension of~$\KK$.  For every place $v\in\mathfrak{M}$, we denote by
$\mathfrak{N}_v$ the set of absolute values~$\abs{\cdot}_{w}$ on~$\FF$ that
extend~$\abs{\cdot}_v$ with weight given by
\[
n_w=\frac{\dim_{\KK_v}(E_w)}{[\FF:\KK]}\, n_v\mbox{,}
\]
where the $E_w$'s are the local Artinian $\KK_v$-algebras in the
decomposition of $\FF\otimes_\KK\KK_v$ from
Lemma~\ref{prop:artiniandecomposition}. Set
$\mathfrak{N}=\bigsqcup_{v\in\mathfrak{M}}\mathfrak{N}_v$. The pair
$(\FF,\mathfrak{N})$ is an adelic field. The adelic fields of this
form are called \emph{adelic field extensions} of
$(\KK,\mathfrak{M})$.
\end{definition}

\begin{remark}
  \label{rem:1}
With  notation as in Lemma~\ref{prop:artiniandecomposition}, 
\[
\dim_{\KK_v}(E_w)=l_{E_w}(E_w)[\FF_w:\KK_v], 
\]
where $l_{E_w}(E_w)$ is the length of $E_{w}$ as a module over itself.
This follows from \cite[Lemma~A.1.3]{Ful98} applied to the morphism
$\KK_v\rightarrow E_w$.  Hence, the weights in Definition~\ref{def:3}
can be alternatively written as
\[
n_w=l_{E_w}(E_w)\frac{[\FF_w:\KK_v]}{[\FF:\KK]}\, n_v\mbox{.}
\]
\end{remark}

\begin{proposition}\label{prop:finiteextension}
  Let $(\KK,\mathfrak{M})$ be an adelic field and $(\FF,\mathfrak{N})$
  an adelic field extension of~$(\KK,\mathfrak{M})$. Then
\begin{enumerate}
\item \label{item:3} the equality $\sum_{w\in \mathfrak{N}_{v}}n_w=n_v$
  holds for every place $v\in\mathfrak{M}$;
\item \label{item:4} if $(\KK,\mathfrak{M})$ satisfies the product
  formula, then $(\FF,\mathfrak{N})$ also does.
\end{enumerate}
\end{proposition}

\begin{proof}
  From the definition of adelic field extension and
  Lemma~\ref{prop:artiniandecomposition}, 
\[
\sum_{w\in\mathfrak{N}_{v}}n_w=\sum_{w\in\mathfrak{N}_{v}}\frac{\dim_{\KK_v}(E_w)}{[\FF:\KK]}
n_v=
\frac{\dim_{\KK_v}(\FF\otimes\KK_v)}{[\FF:\KK]} n_v=n_v\mbox{,}
\]
which proves statement~\eqref{item:3}. To prove the second statement,
let $\alpha\in\FF^\times$ and consider the multiplication map
$\eta_{\alpha}\colon\FF\rightarrow\FF$ given by
$\eta_{\alpha}(x)= \alpha x$.  The norm
$N_{\FF/\KK}(\alpha)\in\KK^\times$ is defined as the determinant of
this $\KK$-linear map.  Moreover, $\eta_{\alpha}$ extends to the
$\KK_v$-linear map
\[
\eta_{\alpha}\otimes
1_{\KK_v}\colon\FF\otimes\KK_v\longrightarrow\FF\otimes\KK_v, 
\]
which has the same determinant.  Using the  decomposition in
\eqref{eq:20}, write 
${\alpha}\otimes 1_{\KK_v}=(\alpha_w)_{w}$ with $\alpha_{w}\in
E_{w}$. Hence $\eta_{\alpha}\otimes
1_{\KK_v}=\bigoplus_{w} \eta_{\alpha_{w}}$ and
\begin{displaymath}
 N_{\FF/\KK}(\alpha)=\det(\eta_{\alpha}\otimes
1_{\KK_v})=\prod_{w\in\mathfrak{N}_{v}}N_{E_w/\KK_v}(\alpha_w)\mbox{.}
\end{displaymath}
By \cite[Chapitre III, \S9.2, Proposition~1]{Bou70}, 
$N_{E_w/\KK_v}(\alpha_w)=N_{\FF_w/\KK_v}(\alpha_w)^{l_{E_w}(E_w)}$.
Moreover, by \cite[VI Proposition~5.6]{Lan02},
\[
N_{\FF_w/\KK_v}(\alpha_w)=\prod_{\sigma}\sigma(\alpha_w)^{[\FF_w:\KK_v]_i}\mbox{,}
\]
where the product is over the different embeddings $\sigma$ of $\FF_w$
in an algebraic closure of~${\KK}_v$, and $[\FF_w:\KK_v]_i$ denotes
the inseparability degree of the extension $\KK_{v}\hookrightarrow
\FF_w$.  Furthermore, the number of such embeddings is equal to the
separability degree $[\FF_w:\KK_v]_s$. For every embedding~$\sigma$,
we have $\abs{\sigma(\alpha_w)}_v=\abs{\alpha}_w$ because the base
field $\KK_{v}$ is complete. Since
$[\FF_w:\KK_v]_i[\FF_w:\KK_v]_s=[\FF_w:\KK_v]$, we get
\[
\abs{N_{\FF/\KK}(\alpha)}_v^{n_v}
=\prod_{w\in\mathfrak{N}_v}\abs{\sigma(\alpha_w)}_v^{l_{E_w}(E_w) [\FF_w:\KK_v] n_v}
=\prod_{w\in\mathfrak{N}_v}\abs{\alpha}_w^{[\FF:\KK]n_w}
\mbox{.}
\]
Since $N_{\FF/\KK}(\alpha)\in\KK^\times$, if $(\KK,\mathfrak{M})$ satisfies the product formula, then
\[
\prod_{w\in\mathfrak{N}}\abs{\alpha}_w^{n_w}
=\bigg(\prod_{v\in\mathfrak{M}}\abs{N_{\FF/\KK}(\alpha)}_v^{n_v}\bigg)^{\frac{1}{[\FF:\KK]}}
=1\mbox{,}
\]
concluding the proof. 
\end{proof}

\begin{example}
\label{exm:3}
  Let $\FF$ be a number field. This is a separable extension of $\QQ$.
  By \cite[Chapitre VI, \S8.5, Corollaire~3]{Bou64}, we have that
  $\FF\otimes\QQ_v\simeq\bigoplus_{w\in\mathfrak{N}_v}\FF_w$ for all
  $v\in \fM_{\QQ}$. Therefore, the weight associated to each place
  $w\in\mathfrak{N}_{v}$ is
\begin{displaymath}
 {n_w=\frac{[\FF_w:\QQ_v]}{[\FF:\QQ]}}.
\end{displaymath}
\end{example}

\begin{example}
\label{exm:4}
  Let $(\K(C), \fM_{\K(C)})$ be the function field of a regular
  projective curve~$C$ over a field $\kappa$ with the structure of adelic
  field as in Example \ref{ex:funtionfield}.  The places of $\K(C)$
  correspond to the closed points of $C$ with absolute values and
  weights given by \eqref{eq:11}.  Let $\FF$ be a finite extension
  of~$\K(C)$ and $\mathfrak{N}$ the set of places of $\FF$ as in
  Definition~\ref{def:3}.  There is a regular projective curve $B$ over
  $\kappa$ and a finite map $\pi\colon B\to C$  such that the
  extension $\K(C)\hookrightarrow\FF$ identifies  with the morphism
  $\pi^*\colon\K(C)\hookrightarrow\K(B)$. For each place $v\in
  \fM_{\K(C)}$, the absolute values of $\FF$ that extend $|\cdot|_{v}$
  are in bijection with the fiber~$\pi^{-1}(v)$.
  
  For each closed point $v\in C$, the integral closure in $\K(B)$ of
  $\cO_{v,C}$ coincides with
  $\cO_{\pi^{-1}(v), B}$, the local ring of $B$ along the fiber
  $\pi^{-1}(v)$. The ring $\cO_{\pi^{-1}(v), B}$ is of finite type
  over $\cO_{v,C}$. With notation as in
  Lemma~\ref{prop:artiniandecomposition}, by \cite[Chapter VI, \S8.5,
  Corollaire~3]{Bou64}, we have $E_w\simeq\FF_w$ for all
  $w\in\mathfrak{N}_v$.
Hence, the weight of  $w$ is given by 
\begin{displaymath}
n_w=\frac{[\FF_w:\K(C)_v]}{[\FF:\K(C)]}[\K(v):\kappa].  
\end{displaymath}

Let  $e({w/v})$ denote the ramification index of $w$ over~$v$. By \cite[Chapter~VI, \S8.5, Corollaire~2]{Bou64}, we have that
$[\FF_w:\K(C)_v]=e({w/v})\,[\K(w):\K(v)]$. Therefore, for each place
$w\in\mathfrak{N}_{v}$, the weight of $w$ can also be expressed as
\[
n_w=\frac{e({w/v})\,[\K(w):\kappa]}{[\FF:\K(C)]}.
\]
\end{example}

Following \cite{BPS14}, a \emph{global field} is a finite extension of
the field of rational numbers or of the function field of a regular
projective curve, with the structure of adelic field described in
Examples \ref{exm:3} and \ref{exm:4}.  For these fields,
Proposition~\ref{prop:finiteextension} is already a known result, see
for instance \cite[Proposition~2.1]{BPRS15}.

By a result of Artin and Whaples, global fields can be characterized
as the adelic fields having an absolute value that is either
Archimedean or associated to a discrete valuation whose residue field
has finite order over the field of constants \cite[Theorems 2 and
3]{ArtinWhaples:acfpf}. 

Function fields of varieties of higher dimension provide examples of
adelic fields satisfying the product formula, and that are not global
fields.

\begin{example}
  \label{exm:5}
Let $\K(S)$ be the function field of  an irreducible normal
variety~$S$ over a field $\kappa$ of dimension $s\ge 1$,
and $E_{1},\dots, E_{s-1}$  nef Cartier divisors on
$S$. Set $  S^{(1)}$ for the set of irreducible hypersurfaces of
$S$. For each $V\in S^{(1)}$, the local ring $\cO_{V,S}$ is a discrete
valuation ring. We associate to $V$ the absolute value and weight
given, for $f\in \K(S)$,  by
\begin{displaymath}
  |f|_{V}=c_{\kappa}^{-\ord_{V}(f)} \and n_{v}=\deg_{E_{1}, \dots, E_{s-1}}(V),
\end{displaymath}
with $c_{\kappa}$ as in \eqref{eq:23}. The set of places
$\fM_{\K(S)}$ is indexed by $S^{(1)}$, and consists of these absolute
values and weights.
For $f\in \K(S)^{\times}$, 
\begin{displaymath}
  \sum_{V\in S^{(1)}}n_{V}\log|f|_{v}= \log(c_{k})   \hspace{-1mm}\sum_{V\in
    S^{(1)}}\deg_{E_{1}, \dots, E_{s-1}}(V) \ord_{V}(f) =
  \deg_{E_{1}, \dots, E_{s-1}}(\div(f))=0,
\end{displaymath}
because the Cartier divisor $\div(f)$ is principal. Hence $(\K(S),
\fM_{\K(S)})$ satisfies the product formula.
\end{example}

\section{Height of cycles} \label{sec:metr-divis-heights}

In this section, we introduce a notion of global height for cycles of
a variety over an adelic field, with respect to a family of metrized
divisors generated by small sections. We also recall the notion of
local height of cycles from \cite[Chapter 1]{BPS14} and give a more
explicit description of this construction in the 0-dimensional case.

Let $(\KK,\fM)$ be an adelic field satisfying the product formula, and
$X$ a normal projective variety over $\KK$. For each place $v\in \fM$,
we denote by $\Xan_{v}$ the $v$-adic analytification of $X$. In the
Archimedean case, if $\KK_v\simeq \CC$, then $\Xan_{v}$ is an analytic
space over~$\CC$ whereas, if $\KK_v\simeq \RR$, then $\Xan_{v}$ is an
analytic space over $\RR$, that is, an analytic space over $\CC$
together with an antilinear involution, as explained in~\cite[Remark
1.1.5]{BPS14}.  In the non-Archimedean case, $\Xan_{v}$ is a Berkovich
space over $\KK_{v}$ as in \cite[\S1.2]{BPS14}.

Fix $v\in \fM$ and set
\begin{displaymath}
 X_{v}=X \times \Spec(\KK_{v}). 
\end{displaymath}
Given a 0-cycle $Y$ of $X_{v}$, a usual construction in Arakelov
geometry associates a signed measure on $X_{v}^{\an}$, denoted by
$\delta_{Y}$, that is supported on $|Y|^{\an}$ and has total mass
equal to $\deg(Y)$, see for instance \cite[Definition 1.3.15]{BPS14}
for the non-Archimedean case.  In what follows, we explicit this
construction.

Let  $q$ be a closed point of $X_{v}$. 
The function field $\K(q)$ is a
finite extension of $\KK_{v}$ and $\deg(q)= [\K(q):\KK_{v}]$. 
If $v$ is
Archimedean, then  $\deg(q)$ is either equal
to 1 or 2. In the first case, the analytification of $q$ is a
point of $\Xan_{v}$ whereas, in the second case, it is a pair of
conjugate points. If $v$ is non-Archimedean, choose an affine open
neighborhood $U=\Spec(A)$ of $q$ and $A\to \K(q)$ the
corresponding morphism of $\KK_{v}$-algebras. The analytification of  $q$ is the point $q^{\an}\in
U^{\an}\subset X_{v}^{\an}$ corresponding to the multiplicative seminorm  given by the composition
\begin{displaymath}
  A \longrightarrow \K(q)
  \xlongrightarrow[]{|\, \cdot\, |} \RR_{\ge0},
\end{displaymath}
where $|\cdot|$ is the unique extension to $\K(q)$
of the absolute value $|\cdot|_{v}$. 

Since the measure $\delta_{q}$ is supported on the point $q^{\an}$ and
has total mass $\deg(q)$, it follows that
\begin{equation}
  \label{eq:37}
  \delta_{q}=[\K(q):\KK_{v}]\, \delta_{q^{\an}}, 
\end{equation}
where $\delta_{q^{\an}}$ denotes the Dirac delta measure on~$q^{\an}$.
For an arbitrary $0$-cycle $Y$ of~$X_{v}$, the signed measure
$\delta_{Y}$ is obtained from \eqref{eq:37} by linearity. It is a
discrete signed measure of total mass equal to $\deg(Y)$.

Let $D$ be a Cartier divisor on $X$. A \emph{metric} on the analytic
line bundle $\mathcal{O}(D)^{\mathrm{an}}_v$ is an assignment that,
to each open subset $U\subset X_{v}^{\an}$ and local section $s$
on~$U$, associates a continuous function
\begin{displaymath}
  \|s(\cdot)\|_{v}\colon U\longrightarrow \RR_{\ge 0}
\end{displaymath}
that is compatible with restrictions to open subsets, vanishes only
when the local section does, and respects multiplication of local sections by
analytic functions, see \cite[Definitions 1.1.1 and 1.3.1]{BPS14}.
This notion allows to define local heights of 0-cycles. 

\begin{definition} \label{def:4} 
Let $D$ be a Cartier divisor on $X$, and $\|\cdot\|_{v}$ a metric on
$\mathcal{O}(D)^{\mathrm{an}}_v$.
For a  $0$-cycle $Y$ of $X_{v}$ and a rational section $s$ of $\cO(D)$ that
is regular and non-vanishing on the support of $Y$, the \emph{local height}
of $Y$ with respect to the pair $(\|\cdot\|_{v},s)$ is defined~as
\[
\h_{\|\cdot\|_{v}}(Y;s)=
-\int_{X_{v}^{\an}}\log\norm{s}_v \, \delta_{Y}\mbox{.}
\]
\end{definition}

We now study the behavior of these objects with respect to adelic field
extensions. Let $(\FF,\fN)$ be an extension of the adelic field
$(\KK,\fM)$ (Definition \ref{def:3}) and fix a place $w\in \fN_{v}$,
so that $\FF_{w}$ is a finite extension of the local field $\KK_{v}$. 
Let $q$ be a closed point of~$X_{v}$ and consider the subscheme
$q_{w}$ of $X_{w}=X \times \Spec(\FF_{w})$ obtained by base change. 
Decompose
\begin{displaymath}
  \K(q)\otimes_{\KK_{v}} \FF_{w}= \bigoplus_{j\in I}G_{j}
\end{displaymath}
as a finite sum of local Artinian $\FF_{w}$-algebras and, for each
$j\in I$, denote by $q_{j}$ the corresponding closed point of $X_{w}$.
Thus the associated cycle is given by
$ [q_{w}]= \sum_{j\in I}l_{G_{j}}(G_{j}) \, {q_{j}} $.  Hence, by
\eqref{eq:37} and Remark~\ref{rem:1},
\begin{displaymath}
\delta_{[q_{w}]}= \sum_{j\in I}\dim_{\FF_{w}}(G_{j})  \, \delta_{q_{j}^{\an}}.
\end{displaymath}

The inclusion $\KK_{v}\hookrightarrow \FF_{w}$ induces a map of the
corresponding analytic spaces 
\begin{equation}
  \label{eq:38}
\pi\colon  X_{w}^{\an}\longrightarrow X_{v}^{\an}.
\end{equation}
In the non-Archimedean case, this map of Berkovich spaces is defined
locally by restricting seminorms.

The following proposition gives the behavior of the measure
associated to a 0-cycle with respect to field extensions. 

\begin{proposition}
  \label{prop:7}
  With notation as above, let $Y$ be a 0-cycle of $X_{v}$ and set
  $Y_{w}$ for the 0-cycle of $X_{w}$ obtained by base change. Then
\begin{displaymath}
  \pi_{*}\, \delta_{Y_{w}}=\delta_{Y}.
\end{displaymath}
\end{proposition}

\begin{proof}
  By the compatibility of the map $\pi$ with restriction to
  subschemes, we have that $\pi(q_{j}^{\an})= q^{\an}$ for all
  $j\in I$. It follows that
\begin{equation*}
  \pi_{*}\, \delta_{[q_{w}]} = \sum_{j\in I}\dim_{\FF_{w}}(G_{j})  \,
  \pi_{*}\, \delta_{q_{j}^{\an}} = 
\bigg( \sum_{j\in I}\dim_{\FF_{w}}(G_{j})  \bigg) \delta_{q^{\an}} = [\K(q):\KK_v]\, \delta_{q^{\an}}
=\delta_{q}.
\end{equation*}  
\end{proof}

Let $D$ be a Cartier divisor on $X$ and $\|\cdot\|_{v}$ a {metric} on
$\mathcal{O}(D)^{\mathrm{an}}_v$. The extension of this metric to a
metric $\|\cdot\|_{w}$ on the analytic line bundle
$\mathcal{O}(D)^{\mathrm{an}}_w$ on $X_{w}^{\an}$ is obtained by
taking the inverse image with respect to the map $\pi$ in
\eqref{eq:38}, that is 
\begin{equation}\label{eq:41}
  \|\cdot\|_{w}= \pi^{*}\,   \|\cdot\|_{v}.
\end{equation}
Proposition \ref{prop:7} implies directly the invariance of the local
height with respect to adelic field extensions. 

\begin{proposition}
\label{prop:6}
With notation as above, let $Y$ be a 0-cycle of $X_{v}$ and $s$ a
rational section of $\cO(D)_{v}^{\an}$ that is regular and non-vanishing on the
support of $Y$. Set $Y_{w}$ and $s_{{w}}=\pi^{*} s$ for the 0-cycle and rational
section obtained by base extension. Then
\begin{displaymath}
\h_{\|\cdot\|_{w}}(Y_{w}, s_{{w}})  =\h_{\|\cdot\|_{v}}(Y, s).  
\end{displaymath}
\end{proposition}

To define global heights of cycles over an adelic field, we consider
adelic families of metrics on the Cartier divisor $D$ satisfying a certain
compatibility condition.

\begin{definition}
  \label{def:5}
  An \emph{(adelic) metric} on $D$ is a collection $\|\cdot\|_{v}$ of
  metrics on $\mathcal{O}(D)^{\mathrm{an}}_v$, ${v\in \fM}$, such
  that, for every point $p\in X(\ol{\KK})$ and a choice of a rational
  section $s$ of~$\cO(D)$ that is regular and non-vanishing at $p$ and
  of an adelic field extension $(\FF,\fN)$ such that $p\in X(\FF)$,
  \begin{equation}\label{eq:14}
    \|s(p_{w}^{\an})\|_{w}=1
  \end{equation}
  for all but a finite number of $w\in \fN$.  We denote by
  $\ol{D}=(D, (\|\cdot\|_{v})_{v\in\fM})$ the corresponding
  \emph{(adelically) metrized divisor} on~$X$.

In addition, $\ol D$ is \emph{semipositive} if each
  of its $v$-adic metrics is semipositive in the sense of
  \cite[Definition~1.4.1]{BPS14}.
\end{definition}

The condition \eqref{eq:14} does not depend on the choice of the
rational section $s$ and of the adelic field extension $(\FF,\fN)$.

\begin{remark}
  \label{rem:4}
  When $\KK$ is a global field, the classical notion of compatibility
  for a collection of metrics $\|\cdot\|_{v}$ on
  $\mathcal{O}(D)^{\mathrm{an}}_v$, ${v\in \fM}$, is that of being
  quasi-algebraic, in the sense that there is an integral model that
  induces all but a finite number of these metrics \cite[Definition
  1.5.13]{BPS14}.

  By Proposition 1.5.14 in \emph{loc. cit.}, a quasi-algebraic
  metrized divisor $\ol{D}$ is adelic in the sense of Definition
  \ref{def:5}. The converse is not true, as it is easy to construct
  toric adelic metrized divisors that are not quasi-algebraic (Remark
  \ref{rem:5}).
\end{remark}

For a 0-cycle $Y$ of $X$ and a place $v\in\fM$, we denote by $Y_{v}$
the $0$-cycle of $X_{v}$ defined by base change. When $Y=p$ is a
closed point of $X$, by Lemma~\ref{prop:artiniandecomposition} applied
to the finite extension $\K(p)$ of $\KK$, the 0-dimensional
subscheme $p_{v}=p \times \Spec(\KK_{v})$ of
$X_{v}$ decomposes as
\[
p_{v}=\Spec(\K(p)\otimes_{\KK}\KK_v)\simeq\coprod_{w\in\mathfrak{N}_v}\Spec(E_w)
\mbox{,}
\]
where the $E_{i}$'s are the local Artinian $\KK_{v}$-algebras in
\eqref{eq:20}. 
Let $q_{w}$, $w\in \fN_{v}$, be the irreducible components of this
subscheme. Then, the associated 0-cycle of $X_{v}$ writes down as
\begin{displaymath}
  [p_{v}]=\sum_{w\in \fN_{v}}l_{E_{w}}(E_{w}) \, q_{w}
\end{displaymath}
and, for each $w\in \fN_{v}$, we have $\K(q_{w})\simeq \K(p)_{w}$.
For an arbitrary $Y$, the 0-cycle $Y_{v}$ is obtained by linearity.

Let $\ol{D}=(D, (\|\cdot\|_{v})_{v\in\fM})$ be a metrized divisor
on~$X$, $Y$ a 0-cycle of $X$ and $s$ a rational section of $\cO(D)$
that is is regular and non-vanishing on the support of $Y$. For each
place $v\in \fM$, we set
  \begin{equation*}
    \label{eq:17}
    \h_{\ol{D},v }(Y;s)=\h_{\|\cdot\|_{v}}(Y_{v};s),
  \end{equation*}
  where $Y_{v}$ is the 0-cycle of $X_{v}$ obtained by base change. The
  condition that $\ol{D}$ is adelic implies that
  $ \h_{\ol{D},v }(Y;s)=0$ for all but a finite number of places.  

  If $s'$ is another rational section of $\cO(D)$ that is regular and
  non-vanishing on $|Y|$, then $s'=fs$ with $f\in \K(X)^{\times}$ and,
  for $v\in \fM$,
  \begin{equation}\label{eq:40}
    \h_{\ol{D},v}(Y;s')=\h_{\ol{D},v}(Y;s)-\log|\gamma|_{v}
  \end{equation}
  where $Y=\sum_{p}\mu_{p}\, p$ and
  $\gamma=\prod_{p}f(p)^{\mu_{p}}\in \KK^{\times}$.

  \begin{definition}
    \label{def:6}
    Let $\ol{D}$ be a metrized divisor on~$X$ and $Y$ a 0-cycle of
    $X$.  The \emph{global height} of $Y$ with respect to $\ol{D}$ is
    defined as
  \begin{equation}\label{eq:16}
    \h_{\ol{D}}(Y)=\sum_{v\in \fM}n_{v}\h_{\ol{D},v}(Y;s),
  \end{equation}
with $s$ a rational section of $\cO(D)$
  that is is regular and non-vanishing on  $|Y|$.
  \end{definition}

  The local heights in \eqref{eq:16} are zero for all but a finite
  number of places, and so this sum is finite. The equality
  \eqref{eq:40} together with the product formula imply that this sum
  does not depend on the rational section $s$.

Given a metrized divisor 
$\ol{D}$  on~$X$ and an adelic field extension
$(\FF,\fN)$, we denote by $\ol{D}_{\FF}$ the metrized divisor on
$X_{\FF}$ obtained by extending the $v$-adic  metrics of $\ol{D}$ as
in~\eqref{eq:41}. 

\begin{proposition}
\label{prop:8}
Let $\ol{D}$ be a metrized divisor on~$X$,  $Y$ a 0-cycle of $X$ and
$(\FF,\fN)$ an adelic field extension of $(\KK,\fM)$. Then
\begin{displaymath}
\h_{\ol{D}_{\FF}}(Y_{\FF})  =\h_{\ol{D}}(Y).  
\end{displaymath}
\end{proposition}

\begin{proof}
Let $s$ be  a rational section of $\cO(D)$
that is is regular and non-vanishing on  $|Y|$ and $v\in \fM$. By  Propositions
\ref{prop:6} and  \ref{prop:finiteextension}\eqref{item:3},
\begin{equation*}
\sum_{w\in \fN_{v}}n_{w}\, 
  \h_{\ol{D}_{\FF}, w}(Y_{\FF}, s)  = 
\sum_{w\in \fN_{v}}n_{w}\, 
  \h_{\ol{D}, v}(Y, s)  = 
n_{v}\, 
  \h_{\ol{D}, v}(Y, s).
\end{equation*}
The statement follows by summing over all the places of $\KK$. 
\end{proof}

Since the global height is invariant under field extension, it induces
a notion of global height for algebraic points, that is, a well-defined function 
\begin{displaymath}
  \h_{\ol{D}}\colon X(\ol{\KK}) \longrightarrow \RR.
\end{displaymath}
When $\KK$ is a global field, this notion coincides with the one in
\cite[Definition~2.2]{BPRS15}.

Now we turn to cycles of arbitrary dimension. Let $V$ be a
$k$-dimensional irreducible subvariety of $X$ and
$\ol{D}_0,\ldots,\ol{D}_{k-1}$ a family of $k$ semipositive metrized
divisors on $X$.  For each place $v\in \fM$, we can associate to this
data a measure on $X_{v}^{\an}$ denoted by
\begin{equation*}
\c_1 (\ol{D_0})\wedge\cdots\wedge\c_1(\ol{D}_{k-1})\wedge\delta_{\Van_v}
\end{equation*}
and called the \emph{$v$-adic Monge-Ampère measure} of $V$ and
$\ol{D}_0,\ldots,\ol{D}_{k-1}$, see \cite[D\'efinition
2.4]{Chambert-Loir:meeb} or \cite[Definition 1.4.6]{BPS14}.  For a
$k$-cycle $Y$ of $X$, this notion extends by linearity to a signed
measure on $X_{v}^{\an}$, denoted by
$\c_1
(\ol{D_0})\wedge\cdots\wedge\c_1(\ol{D}_{k-1})\wedge\delta_{\Yan_v}$.
It is supported on $|Y_{v}|^{\an}$ and has total mass equal to the
degree $\deg_{D_0,\ldots,D_{k-1}}(Y)$.

We recall the notion of local height of cycles from \cite[Definition 1.4.11]{BPS14}.

\begin{definition}\label{def:kcycleheight}
  Let $Y$ be a $k$-cycle of $X$ and, for $i=0,\dots, k$, let
  $(\ol{D}_i, s_{i})$ be a semipositive metrized divisor on $X$ and a
  rational section of $\cO(D_{i})$ such that 
  $\div(s_{0}), \dots, \div(s_{k})$ intersect $Y$ properly (Definition
  \ref{def:8}). For $v\in \fM$, the \emph{local height} of $Y$ with
  respect to $(\ol{D}_0,s_0),\ldots,(\ol{D}_k,s_k)$ is inductively
  defined by the rule
\begin{multline*}
  \mathrm{h}_{\ol{D}_0,\ldots,\ol{D}_k,v}(Y;s_0,\ldots,s_k)=
\mathrm{h}_{\ol{D}_0,\ldots,\ol{D}_{k-1},v}(\div(s_{k})\cdot
Y;s_0,\ldots,s_{k-1}) 
\\ -\int_{\Xan_v}\log\norm{s_k}_{k,v}\on{c}_1(\ol{D}_0)\wedge\ldots\wedge \on{c}_1(\ol{D}_{k-1})\wedge\delta_{Y^{\mathrm{an}}_v}
\end{multline*}
and the convention that the local height of the  cycle $0\in
{Z_{-1}(X)}$ is zero. 
\end{definition}

\begin{remark}
\
  \label{rem:2}
  \begin{enumerate}
  \item \label{item:7} The local height is linear with respect to the
    group structure of $Z_{k}(X)$. In particular, the local heights of
    the cycle $0\in Z_{k}(X)$ are zero. 
  \item \label{item:8} For a $0$-cycle $Y$ of $X$ and $v\in\fM$, the
    $v$-adic Monge-Ampère measure coincides with the measure
    associated to the $0$-cycle $Y_{v}$ of $X_{v}$ at the beginning of
    this section. Hence, Definition \ref{def:kcycleheight} applied to
    a $0$-cycle coincides with Definition~\ref{def:4}.
  \end{enumerate}
\end{remark}

The following notion is the arithmetic analogue of global sections of
a line bundle, and Proposition \ref{prop:recursive} below is an
analogue for local heights of Proposition \ref{prop:5}.

\begin{definition}
  \label{def:14}
Let $\ol{D}=(D,(\norm{\cdot}_v)_{v\in\fM})$ be a metrized divisor on
$X$.  A global section~$s$ of $\cO(D)$ is \emph{$\ol{D}$-small} if,
for all $v\in \fM$,
\begin{displaymath}
  \sup_{q\in X_{v}^{\an}} \|s(q)\|_{v}\le 1.
\end{displaymath}
\end{definition}

\begin{proposition}\label{prop:recursive}
  Let $Y$ be an effective $k$-cycle of $X$ and, for $i=0,\dots, k$,
  let $(\ol{D}_i, s_{i})$ be a semipositive metrized divisor on $X$
  and a rational section of $\cO(D_{i})$ such that $\div(s_{0}),
  \dots, \div(s_{k})$ intersect $Y$ properly and such that $s_k$ is
  $\ol{D}_{k}$-small. Then, for all $v\in \fM$, 
\[
\h_{\ol{D}_0,\ldots,\ol{D}_{k-1},v}(\div(s_k)\cdot Y;s_0,\ldots,s_{k-1})\leq
\h_{\ol{D}_0,\ldots,\ol{D}_k,v}(Y;s_0,\ldots,s_k).
\]
\end{proposition}

\begin{proof}
  Since the cycle $Y$ is effective and the metrized divisors
  $\ol{D}_{i}$ are semipositive, their $v$-adic
  Monge-Ampere measure is a measure, that is, it takes only nonnegative
  values.  Since the global section $s_k$ is $\ol{D}_{k}$-small,
  $\log\norm{s_k(q)}_{k,v}\leq 0$ for all $q\in X_{v}^{\an}$. The
  inequality follows then from the inductive definition of the local
  height.
\end{proof}

Our next step is to define global heights for cycles over an adelic
field. We first state an auxiliary result specifying the behavior of
local heights with respect to change of sections, extending
\eqref{eq:40} to the higher dimensional case. The following lemma and
its proof are similar to \cite[Corollary 3.8]{Gubler:hsmf}.

\begin{lemma}
  \label{lemm:1}
  Let $Y$ be a $k$-cycle of $X$ and $\ol{D}_0, \dots, \ol{D}_{k}$
  semipositive metrized divisors on $X$.  Let $s_{i},s_{i}'$ be
  rational sections of $\cO(D_{i})$, $i=0,\dots, k$, such that both
  $\div(s_{0}), \dots, \div(s_{k})$ and
  $\div(s_{0}'), \dots, \div(s_{k}')$ intersect $Y$ properly. Then
  there exists $\gamma\in \KK^{\times}$ such that, for all $v\in \fM$,
  \begin{equation}\label{eq:24}
  \h_{\ol{D}_0,\ldots,\ol{D}_k,v}(Y;s_0',\ldots,s_k')=
  \h_{\ol{D}_0,\ldots,\ol{D}_k,v}(Y;s_0,\ldots,s_k)-
  \log|\gamma|_{v}.  
  \end{equation}
  \end{lemma}

  \begin{proof}
    Let $s_{i}''$ be rational sections of $\cO(D_{i})$,
    $i=0,\dots, k$, such that $(s_{0}'',\dots, s_{k}'')$ is generic
    enough so that, for every subset $J\subset \{0,\dots, k\}$, the
    family of divisors
    \begin{equation}
      \label{eq:2}
  \{\div(s_{j})\mid j\in J\}\cup\{\div(s_{j}'')\mid j\notin J\}
    \end{equation}
    intersects $Y$ properly.

    We proceed to prove the formula \eqref{eq:24} with the $s_{i}''$'s
    in the place of the $s_{i}'$'s.  Hence, we want to prove that
    there is $\wt \gamma\in \KK^{\times}$ such that, for every
    $v\in \fM$,
\begin{equation}
  \label{eq:4}
    \h_{\ol{D}_0,\ldots,\ol{D}_k,v}(Y;s_0'',\ldots,s_k'')=
  \h_{\ol{D}_0,\ldots,\ol{D}_k,v}(Y;s_0,\ldots,s_k)-
  \log|\wt\gamma|_{v}.  
\end{equation}
To this end, consider first the
particular case when $s_{i}=s_{i}''$, $i=0,\dots, k-1$.  Set
$s_{k}''=fs_{k}$ with $f\in \K(X)^{\times}$, and
$\big( \prod_{i=0}^{k-1}\div(s_{i})\big) \cdot Y=\sum_{p} \mu_{p}\,
p$.
By \cite[Theorem~1.4.17(3)]{BPS14}, the equality \eqref{eq:24} holds
with $\wt \gamma_{k} \in \KK^{\times}$ given by
\begin{displaymath}
  \wt \gamma_{k} = \prod_{p}f(p)^{\mu_{p}}. 
\end{displaymath}
By \cite[Theorem~1.4.17(1)]{BPS14}, the local height is symmetric in
the pairs $(\ol{D_{i}}, s_{i})$.  By the hypothesis \eqref{eq:2}, we
can reorder the metrized line bundles and rational sections, and
iterate the above construction for every $i=0,\dots, k$. This proves
\eqref{eq:4} with $\wt \gamma=\prod_{i=0}^{k}\wt \gamma_{i}$.

Assuming that the $s_{i}''$'s are generic enough so that the condition
in \eqref{eq:2} also holds with the $s_{i}'$'s instead of the
$s_{i}$'s, similarly there exists $\wt \gamma'\in \KK^{\times }$ such
that, for every $v\in \fM$,
\begin{equation}
  \label{eq:6}
      \h_{\ol{D}_0,\ldots,\ol{D}_k,v}(Y;s_0'',\ldots,s_k'')=
  \h_{\ol{D}_0,\ldots,\ol{D}_k,v}(Y;s_0',\ldots,s_k')-
  \log|\wt\gamma'|_{v}.  
\end{equation}
The statement follows by combining \eqref{eq:4} and \eqref{eq:6}. 
  \end{proof}

We consider the following notions of positivity of metrized divisors.

\begin{definition}
\label{def:9}
Let $\ol{D}$ be a metrized divisor on
$X$. 
\begin{enumerate}
\item \label{item:9} $\ol{D}$ is \emph{nef} if $D$ is nef, $\ol{D}$ is
  semipositive, and $\h_{\ol{D}}(p)\ge 0$ for every closed point $p$
  of $X$.
\item \label{item:10}  $\ol D$ is \emph{generated by small sections} if, for
every closed point $p\in X$, there is a $\ol{D}$-small section $s$
such that $p\notin |\div(s)|$.
\end{enumerate}
\end{definition}

  \begin{lemma} \label{lemm:2} Let $Y$ be an effective $k$-cycle of
    $X$ and $(\ol{D}_i, s_{i})$ semipositive metrized divisors on
    $X$ together with rational sections of $\cO(D_{i})$, $i=0,\dots, k$, such
    that $\div(s_{0}), \dots, \div(s_{k})$ intersect $Y$
    properly. Suppose that $\ol{D}_{i}$, $i=1,\dots, k$, are generated
    by small sections. Then there exists $\zeta\in \KK^{\times}$ such
    that, for all $v\in \fM$,
  \begin{displaymath}
  \h_{\ol{D}_0,\ldots,\ol{D}_k,v}(Y;s_0,\ldots,s_k)\ge
  \log|\zeta|_{v}+
  \h_{\ol{D}_{0},v}\bigg(\bigg(\prod_{i=1}^{k}\div(s_{i})\bigg)\cdot Y,
  s_{0}\bigg). 
  \end{displaymath}
  \end{lemma}

  \begin{proof}
    For $k=0$, the statement is obvious, so we only consider the case
    when $k\ge1$. By Lemma \ref{lemm:1}, it is enough to prove the
    statement for any particular choice of rational sections $s_{i}$,
    provided that their associated Cartier divisors intersect $Y$
    properly.

    We can also reduce without loss of generality to the case when
    $Y=V$ is an irreducible variety of dimension $k$. We can then
    choose rational sections $s_{i}$, $i=0,\dots, k$, such that
    $s_{i}$ is $\ol{D}_{i}$-small. By Proposition
    \ref{prop:recursive},
\begin{displaymath}
  \h_{\ol{D}_0,\ldots,\ol{D}_k,v}(V;s_0,\ldots,s_k)\ge 
  \mathrm{h}_{\ol{D}_0,\ldots,\ol{D}_{k-1},v}(\div(s_{k})\cdot
  V;s_0,\ldots,s_{k-1}) .
\end{displaymath}
Since $\div(s_{k})\cdot V$ is an effective $(k-1)$-cycle, the
statement follows by  induction on~$k$. 
  \end{proof}

\begin{prop-def}
\label{prop:4}
Let $Y$ be an effective $k$-cycle of $X$, and
$\ol{D}_0,\dots, \ol{D}_{k}$ semipositive metrized divisors on $X$
such that $\ol{D}_1,\dots, \ol{D}_{k}$ are generated by small
sections.  Let $s_{i}$ be rational sections of $\cO(D_{i})$,
$i=0,\dots, k$, such that $\div(s_{0}), \dots, \div(s_{k})$ intersects
$Y$ properly.  The \emph{global height} of $Y$ with respect to
$\ol{D}_0,\dots, \ol{D_{k}}$ is defined as the sum
\begin{equation}\label{eq:25}
\h_{\ol{D}_0,\ldots,\ol{D}_k}(Y)=    
      \sum_{v\in \fM} n_{v}\, 
\h_{\ol{D}_0,\ldots,\ol{D}_k,v}(Y;s_0,\ldots,s_k).
    \end{equation}
    This sum converges to an element in $\RR \cup\{+\infty\}$, and its
    value does not depend on the choice of the $s_{i}$'s.
\end{prop-def}

\begin{proof}
  The existence of rational sections $s_{i}$ such that
  $\div(s_{0}), \dots, \div(s_{k})$ intersects $Y$ properly follows
  from the moving lemma, with the hypothesis that $X$ is projective.
 
  By Lemma \ref{lemm:2} and the fact that the local heights of
  0-cycles are zero for all but a finite number of places, the
  local heights in \eqref{eq:25} are nonnegative, except for a finite
  number of $v$'s. Hence, the sum converges to an element in
  $\RR\cup\{+\infty\}$. Lemma~\ref{lemm:1} and the product formula
  imply that the value of this sum does not depend on the choice of
  the $s_{i}$'s.
\end{proof}

This definition generalizes the notion of global height of cycles of
varieties over global fields in \cite[\S 1.5]{BPS14}, to cycles of
varieties over an arbitrary adelic field, in the case when the
considered metrized divisors are generated by small sections.

In principle, the sum in \eqref{eq:25} might contain an infinite
number of nonzero terms. Nevertheless, we are not aware of any example
where this phenomenon actually happens. Moreover, for varieties over
global fields, the local heights of a given cycle are zero for all but
a finite number of places \cite[Proposition 1.5.14]{BPS14}, and so
their global height is a real number given as a weighted sum of a
finite number local heights.

In this context, we propose the following question.

\begin{question} \label{ques:1}
    Let $Y$ be a $k$-cycle of $X$ and, for each $i=0,\dots, k$, let
  $(\ol{D}_i, s_{i})$ be a semipositive metrized divisor on $X$ and a
  rational section of $\cO(D_{i})$ such that 
  $\div(s_{0}), \dots, \div(s_{k})$ intersect $Y$ properly. 
Is it true that
\begin{displaymath}
  \mathrm{h}_{\ol{D}_0,\ldots,\ol{D}_k,v}(Y;s_0,\ldots,s_k)=0
\end{displaymath}
for all but a finite number of $v\in \fM_{\KK}$?
\end{question}

A positive answer would imply that, for a variety over an adelic field
and a family of semipositive metrized divisors, the global height of a
cycle is a well-defined real number, given as a weighted sum of a
finite number local heights.

The following results are arithmetic analogues of Proposition
\ref{prop:5} and Corollary~\ref{cor:1}.

\begin{proposition}
  \label{prop:2}
  Let $Y$ be an effective $k$-cycle of $X$, and
  $\ol{D}_0,\dots, \ol{D}_{k}$  semipositive metrized
  divisors on $X$ such that $\ol{D}_{0}$ is nef and
  $\ol{D}_1,\dots, \ol{D}_{k}$ are generated by small sections. Let
  $s_{k}$ be a $\ol{D}_{k}$-small section. Then
\[
0\le \h_{\ol{D}_0,\ldots,\ol{D}_{k-1}}(\div(s_k)\cdot Y)\leq
\h_{\ol{D}_0,\ldots,\ol{D}_k}(Y).
\]
\end{proposition}

\begin{proof}
  We reduce without loss of generality to the case when $Y=V$ is an
  irreducible subvariety of dimension $k$.  If
  $V\subset |\div(s_{k})|$, the first inequality is clear. For the
  second inequality, we choose rational sections $s_{i}$,
  $i=0,\dots, k-1$, and $s_{k}'$ such that
  $\div(s_{0}), \dots, \div(s_{k-1}), \div(s_{k}')$ intersect $Y$
  properly.  Using Lemmas \ref{lemm:1} and \ref{lemm:2}, the product
  formula and the fact that $\ol{D}_{0}$ is nef, we deduce that
  $\h_{\ol{D}_0,\ldots,\ol{D}_k}(Y)\ge 0$.

  Otherwise, $V\not\subset |\div(s_{k})|$ and we choose rational
  sections $s_{i}$, $i=0,\dots, k-1$, such that
  $\div(s_{0}), \dots, \div(s_{k})$ intersect $Y$ properly. The first
  inequality follows by applying the argument above to
  $\div(s_k)\cdot Y$, whereas the second one is given by
  Proposition~\ref{prop:recursive}.
\end{proof}

\begin{corollary}
  \label{cor:2}
  Let $\ol D_{0},\dots, \ol D_{n}$ be semipositive metrized divisors
  on $X$ such that $\ol{D}_{0}$ is nef and
  $\ol{D}_1,\dots, \ol{D}_{n}$ are generated by small sections. Let
  $s_{i}$ be $\ol{D}_{i}$-small sections, $i=1,\dots, n$. Then
\begin{displaymath}
 0\leq \h_{\ol{D}_{0}}\bigg(\prod_{i=1}^{n}\div(s_{i}) \bigg)\le \h_{\ol{D}_{0},\dots, \ol{D}_{n}}(X).
\end{displaymath}
\end{corollary}

\section{Metrics and heights on toric varieties} \label{sec:toric-varieties}

In this section, we recall the necessary background on the arithmetic
geometry of toric varieties following \cite{BPS14, BMPS12}. In the
second part of \S\ref{sec:intersection-theory}, we presented 
elements of the algebraic geometry of toric varieties over a
field. In the sequel, we will freely use the notation introduced
therein. 

Let $(\KK,\fM)$ be an adelic field satisfying the product formula. Let
$M\simeq \ZZ^{n}$ be a lattice and $\TT\simeq\GG_{\text{m},\KK}^n$ its
associated torus over $\KK$ as in \eqref{eq:29}.  For
$v\in\mathfrak{M}$, we denote by $\mathbb{T}^{\mathrm{an}}_v$ the
$v$-adic analytification of $\mathbb{T}$, and by $\mathbb{S}_v$ its
compact torus. In the Archimedean case, $\SS_v$ is isomorphic to
the polycircle $(S^ 1)^n$, whereas in the non-Archimedean case, it is
a compact analytic group, see~\cite[\S4.2]{BPS14} for a description.
Moreover, there is a map
$\val_v\colon \mathbb{T}^{\mathrm{an}}_v \rightarrow N_{\mathbb{R}}$
defined, in a given splitting, as
\begin{displaymath}
\val_{v}(x_1,\ldots,x_n) = (-\log\abs{x_1}_v,\ldots,-\log\abs{x_n}_v).  
\end{displaymath}
This map does not depend on the choice of the splitting, and the
compact torus~$\mathbb{S}_v$ coincides with its fiber over the point
$0\in N_\RR$.

Let $X$ be a projective toric variety with torus $\TT$ given by a
regular complete fan~$\Sigma$ on $N_{\RR}$, and $D$ a toric Cartier
divisor on $X$ given by a virtual support function $\Psi_{D}$
on~$\Sigma$. Recall that $X$ contains $\TT$ as a dense open
subset. Let $\|\cdot\|_{v}$ be a toric $v$-adic metric on $D$, that
is, a metric on the analytic line bundle $\cO(D)^{\an}_{v}$ that is
invariant under the action of $\SS_v$.  The associated \emph{$v$-adic
  metric function} is the continuous function
$\psi_{\norm{\cdot}_v}\colon N_{\mathbb{R}} \rightarrow \mathbb{R}$
given by
\begin{equation}
  \label{eq:35}
\psi_{\|\cdot\|_{v}} (u)=\log\norm{s_D(p)}_v,    
\end{equation}
for any $p\in \TT_{v}^{\an}$ with $\val_v(p)= u$ and where $s_{D}$ is
the distinguished rational section of~$\cO(D)$. This function
satisfies that $|\psi_{\norm{\cdot}_v}-\Psi_{D}|$ is bounded on
$N_{\RR}$ and moreover, this difference extends to a continuous
function on $N_{\Sigma}$, the compactification of $N_{\RR}$ induced by
the fan $\Sigma$. Indeed, the assignment
\begin{equation}\label{eq:30}
\norm{\cdot}_v\longmapsto \psi_{\norm{\cdot}_v}  
\end{equation}
is a one-to-one correspondence between the set of toric $v$-adic
metrics on $D$ and the set of such continuous functions on $N_{\RR}$
\cite[Proposition~4.3.10]{BPS14}.  In particular, the toric $v$-adic
metric on $D$ associated to the virtual support function $\Psi_D$ is
called the \emph{canonical} $v$-adic toric metric of $D$ and is
denoted by~$\norm{\cdot}_{v,\mathrm{can}}$.

Furthermore, when $\|\cdot\|_{v}$ is semipositive,
$\psi_{\|\cdot\|_{v}}$ is a concave function and it verifies that
$|\psi_{\norm{\cdot}_v}-\Psi_{D}|$ is bounded on $N_{\RR}$, and the
assignment in \eqref{eq:30} gives a one-to-one correspondence between
the set of semipositive toric $v$-adic metrics on $D$ and the set of
such concave functions on $N_{\RR}$. 

When $\|\cdot\|_{v}$ is semipositive, we also consider a continuous
concave function on the polytope
$\vartheta_{\|\cdot\|_{v}}\colon \Delta_{D}\to \RR$ defined as the
Legendre-Fenchel dual of $\psi_{\|\cdot\|_{v}}$, that is
\begin{displaymath}
  \vartheta_{\|\cdot\|_{v}}(x)= \inf_{u\in N_{\RR}} \langle x,u\rangle
  - \psi_{\|\cdot\|_{v}}(u). 
\end{displaymath}
The assignment $\|\cdot\|_{v}\mapsto \vartheta_{\|\cdot\|_{v}}$ is a
one-to-one correspondence between the set of semipositive toric
$v$-adic metrics on $D$ and that of continuous concave functions
on~$\Delta_{D}$.  Under this assignment, the canonical $v$-adic toric
metric of $D$ corresponds to the zero function on $\Delta_{D}$.

\begin{definition}\label{def:10}
  An \emph{(adelic) toric metric} on $D$ is a collection of toric
  $v$-adic metrics $(\|\cdot\|_{v})_{v\in \fM}$, such that
  $\|\cdot\|_{v}=\norm{\cdot}_{v,\mathrm{can}}$ for all but a finite
  number of $v\in \fM$.  We denote by
  $\ol D= (D, (\|\cdot\|_{v})_{v\in \fM})$ the corresponding \emph{(adelic)
  toric metrized divisor} on~$X$.
\end{definition}

\begin{example}
  \label{exm:6}
  The collection $(\|\cdot\|_{v,\can})_{v\in \fM}$ of $v$-adic toric
  metrics on $D$ is adelic in the sense of Definition \ref{def:10}.
  We denote by $\ol{D}^{\can}$ the corresponding \emph{canonical}
  toric metrized divisor on $X$.
\end{example}

Let $\ol{D}$ be a toric metrized divisor on $X$. For each $v\in \fM$,
we set
\begin{displaymath}
  \psi_{\ol{D}, v}=\psi_{\|\cdot\|_{v}} \and \vartheta_{\ol{D}, v}=\vartheta_{\|\cdot\|_{v}}
\end{displaymath}
for the associated \emph{$v$-adic metric function} and \emph{$v$-adic
  roof function}, respectively.

\begin{proposition}
  \label{prop:3}
  Let $\ol{D}=(D, (\|\cdot\|_{v})_{v\in \fM})$ be toric divisor
  together with a collection of toric $v$-adic metrics. If $\ol{D}$ is
  adelic in the sense of Definition \ref{def:10}, then it is also
  adelic in the sense of Definition \ref{def:5}. Moreover, both
  definitions coincide in the semipositive case.
\end{proposition}

\begin{proof}
Let  $p\in X(\ol{\KK})$ and choose  an adelic
  field extension $(\FF,\fN)$ such that $p\in X(\FF)$. Then $p_{\FF}$
  is a rational point of $X_{\FF}$ and the inclusion 
  \begin{displaymath}
   \iota\colon  p_{\FF}\longhookrightarrow X_{\FF}
  \end{displaymath}
is an equivariant map. Hence the inverse image $\iota^{*}\ol{D}$ is an
adelic toric metric on $p_{\FF}$ and so, for $w\in \fN$, 
\begin{displaymath}
  \log\|p_{\FF}\|_{w}= \psi_{\iota^{*}\ol{D},w}(0), 
\end{displaymath}
and this quantity vanishes for all but the finite number of $w\in \fN$
such that $\|\cdot\|_{w}$ is not the canonical metric. Since this
holds for all $p\in X(\ol{\KK})$, we conclude that $\ol{D}$ is adelic
in the sense of Definition \ref{def:5}.

For the second statement, assume that $\ol D$ is semipositive and adelic
in the sense of Definition \ref{def:5}. 
Let $x_{i}\in M$, $i=1,\dots, s$, be the vertices of the lattice
polytope $\Delta_{D}$. By \cite[Example
2.5.13]{BPS14}, for each $i$ there is an $n$-dimensional  cone $\sigma_{i}\in
\Sigma$ corresponding to $x_{i}$ under the Legendre-Fenchel
correspondence, $i=1,\dots, s$. 
Each of these  $n$-dimensional cones corresponds to a 0-dimensional
orbit $p_{i}$ of $X$. Denote by $\iota_{i}\colon p_{i}\hookrightarrow
X$ the inclusion of this orbit. 

Fix $1\le i\le s$. Modulo a translation, we can assume without loss of
generality that $x_{i}=0$. By \cite[Proposition 4.8.9]{BPS14}, for
$v\in \fM$,
\begin{displaymath}
  \vartheta_{\ol{D}, v}(x_{i})= \vartheta_{\iota_{i}^{*}\ol{D}, v}(0)=-\log\|s_{D}(p_{i})\|_{v}. 
\end{displaymath}
Hence $\vartheta_{\ol{D}, v}(x_{i})=0$ for all but a finite number of
$v$'s. 

On the other hand, let $x_{0}$ be the distinguished point of $X$,
which coincides with the neutral element of $\TT$, and denote by 
$\iota_{0}\colon x_{0}\hookrightarrow X$ its inclusion. By
\cite[Proposition~4.8.10]{BPS14},
\begin{displaymath}
\max_{x\in \Delta_{D}}  \vartheta_{\ol{D}, v}(x)= \vartheta_{\iota_{0}^{*}\ol{D}, v}(0)=-\log\|s_{D}(x_{0})\|_{v}. 
\end{displaymath}
Hence $\max_{x\in \Delta_{D}} \vartheta_{\ol{D}, v}(x)=0$ for all but a
finite number of $v$'s. 

For every $v\in \fM$ such that $\vartheta_{\ol{D},v}(x_{i})=0$ for all
$i$ and $\max_{x\in \Delta_{D}} \vartheta_{\ol{D}, v}(x)=0$, we have
that $ \vartheta_{\ol{D}, v}\equiv 0$ because this local roof function
is a concave function on $\Delta_{D}$.  Hence, $\|\cdot\|_{v}$
coincides with the $v$-adic canonical metric of $D$ for all these
places.
\end{proof}

\begin{remark}
  \label{rem:5}
  In the general non-semipositive case, Definitions \ref{def:10} and
  \ref{def:5} do not coincide. For instance, when $X=\PP^{1}_{\KK}$, a
  collection of metrics $\|\cdot\|_{v}$, $v\in \fM$, satisfies
  Definition \ref{def:5} if and only if its associated metric
  functions satisfy that
  \begin{displaymath}
    \psi_{\ol{D},v}(0)= 0 \and \lim_{u\to \pm\infty}\psi_{\ol{D},
      v}(u)- \Psi_{{D}}(u) =0
  \end{displaymath}
for all but a finite number of places. 
In the absence of convexity, these conditions do not imply that
$\psi_{\ol{D}, v}= \Psi_{D}$ for all but a finite number of places. 
\end{remark}

A classical example of toric metrized divisors are those given by the
inverse image of an equivariant map to a projective space equipped
with the canonical metric on its universal line bundle. Below we
describe this example and we refer to \cite[Example 5.1.16]{BPS14} for
the technical details.

Let $\bfm=(m_{0}, \dots, m_{r})\in M^{r+1}$ and
$\bfalpha=(\alpha_{0},\dots, \alpha_{r})\in (\KK^{\times})^{r+1}$, with
$r\ge 0$. The monomial map associated to this data is defined as
\begin{equation}
  \label{eq:42}
\varphi_{\bfm,\bfalpha} \colon \TT \longrightarrow
\PP^{r}_{\KK},
  \quad  p \longmapsto (\alpha_{ 0} \chi^{ m_{0}}(p):\dots: \alpha_{ r} \chi^{ m_{r}}(p)).
\end{equation}
For a toric variety $X$ with torus $\TT$ corresponding to fan that is
compatible with the polytope
$\Delta=\conv(m_{0}, \dots, m_{r})\subset M_{\RR}$, this extends to an
equivariant map $X\to \PP^{r}_{\KK}$, also denoted
by~$\varphi_{\bfm,\bfalpha} $.

\begin{example}\label{ex:canonicpullback0}
  With notation as above, let $\ol{E}^{\can}$ be the divisor of the
  hyperplane at infinity of $\PP^{r}_{\KK}$, equipped with the
  canonical metric at all places.  Then
  $D=\varphi_{\bfm,\bfalpha}^{*}E$ is the nef toric Cartier divisor on
  $X$ corresponding to the translated polytope $\Delta-m_{0}$.  We
  consider the semipositive toric metrized divisor
  $\ol{D}= \varphi_{\bfm,\bfalpha}^{*}\ol E$ on~$X$.

For each $v\in \fM$, the $v$-adic metric function of $\ol{D}$ is given by
\begin{displaymath}
\psi_{\ol{D},v}\colon N_{\mathbb{R}} \longrightarrow\RR, \quad u
\longmapsto \min_{ 0\le j\le r}\Big( \langle  m_{j}-m_{0},
u\rangle-\log\Big|\frac{\alpha_{j}}{\alpha_{0}}\Big|_v\Big) .
\end{displaymath}
The polytope corresponding to $D$ is $\Delta-m_{0}$ and, for each
$v\in \fM$, the $v$-adic roof function of $\ol{D}$ is given by
\begin{displaymath}
 \vartheta_{\ol{D},v}(x)= \max_{\bflambda}
\sum_{j=0}^{r}\lambda_{j}\log{|\alpha_{j}|_{v}}-\log|\alpha_{0}|_{v},
\end{displaymath}
the maximum being over the vectors
$\bflambda=(\lambda_{0},\dots,\lambda_{r})\in \RR_{\ge0}^{r+1}$ with
$\sum_{j=0}^{r} \lambda_{j}=1$ such that
$\sum_{j=0}^{r}\lambda_{j}(m_{j}-m_{0})=x$.  In other words, this is the
piecewise affine concave function on $\Delta-m_{0}$ parametrizing the upper
envelope of the extended polytope
\begin{displaymath}
  \conv\big(( m_{j}-m_{0},\log|\alpha_{j}/\alpha_{0}|_{v})_{ 0\le j\le r}\big)  \subset M_{\RR}\times \RR.
\end{displaymath}
\end{example}

\begin{definition}
\label{def:11}
For $i=0,\dots, n$, let $g_i:\Delta_i\rightarrow\RR$ be a concave
function on a convex body $\Delta_i\subset M_{\mathbb{R}}$.  The
\emph{mixed integral} of $g_0,\ldots,g_n$ is defined as
\[
\MI_M(g_0,\ldots,g_n)=
\displaystyle\sum_{j=0}^n (-1)^{n-j}
\displaystyle\sum_{0\leq i_0<\cdots<i_j\leq n}
\int_{\Delta_{i_0}+\cdots+\Delta_{i_j}}g_{i_0}\boxplus\cdots\boxplus g_{i_j}\,\mathrm{d}\vol_M,
\]
where $\Delta_{i_0}+\cdots+\Delta_{i_j}$ denotes the Minkowski sum of
polytopes, and $g_{i_0}\boxplus\cdots\boxplus g_{i_j}$ the
sup-convolution of concave function, which is the function on
$\Delta_{i_0}+\cdots+\Delta_{i_j}$ defined as
\begin{displaymath}
g_{i_0}\boxplus\cdots\boxplus g_{i_j}(x)= \sup \big( g_{i_0}( x_{i_0})+\cdots+g_{i_j}(x_{i_j})\big),  
\end{displaymath}
where the supremum is taken over $x_{i_{l}}\in \Delta_{i_{l}}$,
$l=0,\dots, j$, such
that $x_{i_0}+\cdots+ x_{i_j}= x$. 
\end{definition}

The mixed integral is symmetric and additive in each variable with
respect to the sup-convolution.  Moreover, for a concave function
$g\colon \Delta\rightarrow\RR$ on a convex body~$\Delta$, we have
$\MI_M(g,\ldots,g)=(n+1)!\int_{\Delta}g\,\mathrm{d}\vol_M$, see
\cite[\S8]{PhilipponSombra:rBKe} for details.

The following is a restricted version of a result by Burgos Gil,
Philippon and the second author, giving the global height of a toric
variety with respect to a family of semipositive toric metrized
divisors in terms of the mixed integrals of the associated local roof
functions \cite[Theorem 5.2.5]{BPS14}.

\begin{theorem}\label{thm:heightoftoricvariety}
Let $\ol{D}_i$, $i=0,\ldots,n$, be semipositive 
toric metrized  divisors on $X$ such that $\ol{D}_{1}, \dots,
\ol{D}_{n}$ are generated by small sections.  
Then 
\begin{equation}
  \label{eq:43}
\mathrm{h}_{\ol{D}_0,\ldots,\ol{D}_n}(X)=
\displaystyle\sum_{v\in\mathfrak{M}}
n_v\MI_M(\vartheta_{\ol{D}_0,v},\ldots,\vartheta_{\ol{D}_n,v}) .  
\end{equation}
\end{theorem}

\begin{remark}
  \label{rem:6}
  The result in \cite[Theorem 5.2.5]{BPS14} is more general.
  Given semipositive toric metrized divisors $\ol{D}_i$,
  $i=0,\ldots,n$, and rational sections $s_{i}$ such that
  $\div(s_{0}),\dots, \div(s_{n})$ intersect $X$ properly, the
  corresponding local heights are zero except for a finite number of
  places, and the formula \eqref{eq:43} holds without any extra
  positivity assumption.
\end{remark}

\section{Proof of Theorem \ref{thm:1}} \label{sec:theorems}

Let $(\KK,\fM)$ be an adelic field satisfying the product formula. Let
$f\in \KK[M]$ be a Laurent polynomial and $\Delta\subset M_{\RR}$ its
Newton polytope. Let $X$ be a projective toric variety over $\KK$
given by a fan on $N_{\RR}$ that is compatible with $\Delta$, and $D$
the Cartier divisor on $X$ given by this polytope.  To prove
Theorem~\ref{thm:1}, we first construct a toric metric on $D$ such
that the associated toric metrized divisor $\ol{D}$ is semipositive
and generated by small sections, and the global section of $\cO(D)$
associated to $f$ is $\ol{D}$-small. We obtain this metrized divisor
as the inverse image of a metrized divisor on a projective space.

For $r\ge 0$, let $\PP^r_\KK$ be the $r$-dimensional projective space
over $\KK$ and $E$ the divisor of the hyperplane at infinity.  We
denote by $\ol{E}$ this Cartier divisor equipped with the
$\ell^{1}$-norm at the Archimedean places, and the canonical one at
the non-Archimedean ones. This metric is defined, for
$p=(p_{0}:\dots:p_{s})\in \PP^{s}_{\KK}(\ol{\KK}_{v})$ and a global section
$s$ of~$\cO(E)$ corresponding to a linear form
$\rho_{s}\in \KK[x_{0},\dots, x_{s}]$, by
\begin{equation} \label{eq:45}
    \|s(p)\|_{v}=
\begin{cases}
\displaystyle{\frac{|\rho_{s}(p_{0},\dots, p_{s})|_{v}}{\sum_{j}|p_{j}|_{v}}} & \text{ if } v \text{ is
  Archimedean}, \\[4mm]
\displaystyle{\frac{|\rho_{s}(p_{0},\dots, p_{s})|_{v}}{\max_{j}|p_{j}|_{v}}} & \text{ if } v \text{ is non-Archimedean}, 
\end{cases}
\end{equation}
The projective space $\PP^{r}_{\KK}$ has a standard structure of toric
variety with torus $\GmK^{r}$, included via the map
$(z_{1},\dots, z_{r})\mapsto (1:z_{1}:\dots:z_{r})$. Thus $\ol{E}$ is
a toric metrized divisor. It is a particular case of the weighted
$\ell^{p}$-metrized divisors on toric varieties studied in
\cite[\S5.2]{BPS:smthf}.

The following result summarizes the  basic properties of  this
toric metrized divisor and its combinatorial data.

\begin{proposition}
  \label{prop:9}
The toric metrized divisor  $\ol{E}$ on $\PP^r_{\KK}$ is  semipositive and generated by small sections.
For $v\in \fM$, its $v$-adic metric function is given, for $\bfu=(u_{1},\dots, u_{r})\in \RR^{r}$, by 
\begin{equation}\label{eq:44}
\psi_{\ol{E},v}(\bfu)=
\begin{cases}
-\log\Big(\displaystyle 1+\sum_{j=1}^{r}\e^{-u_j}\Big) & \mbox{if $v$ is Archimedean,}\\
\phantom{-}\min( 0,u_1,\ldots,u_r) & \mbox{if $v$ is non-Archimedean.}  
\end{cases}
\end{equation}
The polytope corresponding to $E$ is the standard simplex $\Delta^{r}$
of $\RR^{r}$. For $v\in \fM$, the $v$-adic roof function of
$\ol{E}$ is given, for $\bfx=(x_{1},\dots, x_{r})\in \Delta^{r}$, by
\begin{displaymath}
\vartheta_{\ol{E},v}(\bfx)=
\begin{cases}
\displaystyle{-\sum_{j=0}^rx_j\log(x_j)}
 & \mbox{if $v$ is Archimedean,}\\
\phantom{-} 0 & \mbox{if $v$ is non-Archimedean,}  
\end{cases}
\end{displaymath}
with  $x_0=1-\sum_{j=1}^rx_j$.
\end{proposition}

\begin{proof}
  The distinguished rational section of the line bundle $\cO(E)$ corresponds to the linear form
  $x_{0}\in \KK[x_{0},\dots, x_{r}]$.  Hence, for an Archimedean place
  $v$ and a point $\bfz=(z_{1},\dots, z_{r})\in \GmK^{r}(\ol{\KK}_{v})$,
\begin{displaymath}
  \psi_{\ol{E},v}(\val_{v}(\bfz))= \log\|s_{E}(\bfz)\|_{v}= -\log\Big(
1+  \sum_{j=1}^{r} |z_{j}|  \Big), 
\end{displaymath}
which gives the expression in \eqref{eq:44} for this case. The
non-Archimedean case is done similarly.  We can easily check that these
metric functions are concave. In the Archimedean case, this can be
done by computing its Hessian and verifying that it is nonpositive
and, in the non-Archimedean case, it is immediate from its
expression. Hence, $\ol{E}$ is semipositive.

Set $s_{j}$ for the global sections corresponding to the linear forms
$x_{j}\in \KK[x_{0},\dots, x_{r}]$, $j=0,\dots, r$. We have that
$\bigcap_{j=0}^{r}|\div(s_{j})|=\emptyset$, and so this is a set of
generating global sections. It follows from the definition of the metric in
\eqref{eq:45} that these global sections are $\ol{E}$-small. Hence, $\ol{E}$
is generated by small sections.

The fact that the polytope corresponding to $E$ is the standard
simplex is classical, see for instance \cite[page 27]{Ful93}.  When
$v$ is Archimedean, the $v$-adic roof function can be computed
similarly as the one for the Fubini-Study metric in
\cite[Example~2.4.3]{BPS14}. When $v$ is non-Archimedean, $v$-adic
roof function is zero, because the metric $\|\cdot\|_{v}$ is
canonical.
\end{proof}

Set $r\ge 0$. Take $\bfm=(m_{0},\dots, m_{r})\in M^{r+1}$ and
$\bfalpha=(\alpha_{0},\dots, \alpha_{r}) \in (\KK^{\times})^{r+1}$,
and consider the polytope
$\Delta=\conv(m_{0},\dots, m_{r})\subset M_{\RR}$.  Let $X$ be a
projective toric variety over $\KK$ given by a fan on $N_{\RR}$ that
is compatible with $\Delta$.  Let
$ \varphi_{\bfm,\bfalpha} \colon \TT \rightarrow \PP^{r}_{\KK}$ be the
monomial map in \eqref{eq:42} and set
\begin{equation*}
  D_{\bfm}=\div(\chi^{-m_{0}})+\varphi_{\bfm,\bfalpha} ^{*}{E},
\end{equation*}
which coincides with the Cartier divisor on $X$ corresponding to
$\Delta$. For each $v\in \fM$, we consider the metric on
$\cO(D_{\bfm})^{\an}_{v}\simeq\cO(\varphi_{\bfm,\bfalpha}
^{*}{E})^{\an}_{v}$ defined by
\begin{equation} \label{eq:51}
  \|\cdot\|_{\bfm,\bfalpha,v}= |\alpha_{0}|_{v}^{-1} \varphi_{\bfm,\bfalpha}
^{*}\|\cdot\|_{\ol{E}, v},
\end{equation}
the homothecy by $|\alpha_{0} |_{v}$ of the inverse image by
$ \varphi_{\bfm,\bfalpha}$ of the $v$-adic metric of $\ol{E}$.
We then set 
\begin{equation}
  \label{eq:46}
  \ol{D}_{\bfm,\bfalpha}= (D_{\bfm},(\|\cdot\|_{\bfm,\bfalpha,v})_{v\in \fM}).
\end{equation}
Since  $ \varphi_{\bfm,\bfalpha}$ is an equivariant map and $\ol{E}$ is
toric, this is a toric metrized divisor on $X$.

\begin{proposition}
\label{prop:10}
The toric metrized divisor $\ol{D}=\ol{D}_{\bfm,\bfalpha}$ on $X$ is semipositive and
generated by small sections. For $v\in \fM$, its $v$-adic  metric is given, for
$p\in \TT(\ol{\KK}_{v})$, by
\begin{equation} \label{eq:48}
    \|s_{D}(p)\|_{v}=
\begin{cases}
{\bigg({\displaystyle \sum_{j=0}^{r}|\alpha_{j} \chi^{m_{j}}(p)|_{v}}\bigg)^{-1}} & \text{ if } v \text{ is
  Archimedean}, \\[4mm]
\Big(\displaystyle \max_{0\le j\le r}|\alpha_{j} \chi^{m_{j}}(p)|_{v}\Big)^{-1}  & \text{ if } v \text{ is non-Archimedean}.
\end{cases}
\end{equation}
The $v$-adic metric function of $\ol{D}$ is given, for $u\in N_{\RR}$,
by
\begin{equation}\label{eq:50}
\psi_{\ol{D},v}(u)=
\begin{cases}
-\log\bigg(\displaystyle\sum_{j=0}^{r}\Abs{{\alpha_j}}_v\e^{-\langle m_{j},u\rangle}\bigg) & \mbox{if $v$ is Archimedean,}\\
\displaystyle  \phantom{-} \min_{ 0\le j\le r}\langle  m_{j},
u\rangle-\log\Abs{{\alpha_{j}}}_v& \mbox{if $v$ is non-Archimedean},
\end{cases}
\end{equation}
and the $v$-adic roof function of
$\ol{D}$ is given, for $x\in  \Delta$, by
\begin{equation}\label{eq:49}
\vartheta_{\ol{D},v}(x)=
\begin{cases}
\displaystyle  \max_{\bflambda}
\sum_{j=0}^{r}\lambda_{j}\log \Big(
\frac{|\alpha_{j}|_{v}}{\lambda_{j}}\Big)& \mbox{if $v$ is Archimedean,}\\
\displaystyle  \max_{\bflambda}
\sum_{j=0}^{r}\lambda_{j}\log{|\alpha_{j}|_{v}} & \mbox{if $v$ is non-Archimedean,}\\
\end{cases}
\end{equation}
the maximum being over the vectors
$\bflambda=(\lambda_{0},\dots,\lambda_{r})\in \RR_{\ge0}^{r+1}$ with
$\sum_{j=0}^{r} \lambda_{j}=1$ such that
$\sum_{j=0}^{r}\lambda_{j}m_{j}=x$. 
\end{proposition}

\begin{proof}
  Set $\ol{D'}=\varphi_{\bfm,\bfalpha} ^{*}\ol{E}$ for short. This is
  a toric metrized divisor on $X$ that is semipositive and generated
  by small sections, due to Proposition~\ref{prop:9} and the
  preservation of these properties under inverse image.  Since the
  $v$-adic metrics of $ \ol{D}$ are homothecies of those of $\ol{D'}$,
  it follows that $\ol{D}$ is semipositive too. Moreover, a global section
  $\varsigma$ of $\cO(D')\simeq\cO(D)$ is $\ol{D'}$-small if and only
  if the global section $\alpha_{0}\, \varsigma$ is $\ol{D}$-small. It
  follows that $\ol{D}$ is also generated by small sections.

  Using \eqref{eq:45} and the definition of the monomial map
  $\varphi_{\bfm,\bfalpha}$, for $v\in \fM$, the
  $v$-adic metric of $\ol{D'}$ is given, for $p\in \TT(\ol{\KK}_{v})$,
  by
\begin{equation*}
    \|s_{D'}(p)\|_{v}=
\begin{cases}
{\bigg({\displaystyle  \sum_{j=0}^{r}\bigg|\frac{\alpha_{j}}{\alpha_{0}} \chi^{m_{j}-m_{0}}(p)\bigg|_{v}}\bigg)^{-1}} & \text{ if } v \text{ is
  Archimedean}, \\[4mm]
\bigg(\displaystyle \max_{0\le j\le r}\bigg|\frac{\alpha_{j}}{\alpha_{0}} \chi^{m_{j}-m_{0}}(p)\bigg|_{v}\bigg)^{-1}  & \text{ if } v \text{ is non-Archimedean}.
\end{cases}
\end{equation*}
Since $D=\div(\chi^{-m_{0}})+D'$, their distinguished
rational sections are related by $s_{D}=\chi^{-m_{0}} s_{D'}$. It
follows from \eqref{eq:51} that
\begin{displaymath}
    \|s_{D}(p)\|_{v}=  |\alpha_{0}|_{v}^{-1}  |\chi^{-m_{0}}(p)|_{v}\,
    \|s_{D'}(p)\|_{v},
\end{displaymath}
which implies the formulae in \eqref{eq:48}. 
As a consequence, we obtain also the expressions for the $v$-adic
metric functions of $\ol{D}$.

For its roof
function, consider first the linear map $H\colon N_{\RR}\to \RR^{r+1}$
given, for $u\in N_{\RR}$, by
$H(u)= (\langle m_{0},u\rangle, \dots, \langle m_{r},u\rangle)$.
For each place $v$, consider the concave function
$g_{v}\colon \RR^{r+1}\to \RR$ given, for $\bfnu\in\RR^{r+1}$, by
\begin{displaymath}
  g_{v}(\bfnu)=
  \begin{cases}
 -\log\bigg(\displaystyle  \sum_{j=0}^{r}|\alpha_{j}|_{v}\e^{-\nu_{j}}\bigg)  & \text{ if } v \text{ is
  Archimedean}, \\
\displaystyle\phantom{-} \min_{0\le j\le r} \nu_{j}-\log|\alpha_{j}|_{v}  & \text{ if } v \text{ is
  non-Archimedean}.
  \end{cases}
\end{displaymath}
Notice that $\psi_{\ol{D},v}=H^{*}g_{v}$.  The domain of the
Legendre-Fenchel dual of $g_{v}$ is the simplex $S$ given as the
convex hull of the vectors in the standard basis of $\RR^{r+1}$.  This
Legendre-Fenchel dual is given, for $\bflambda\in S$, by
\begin{displaymath}
  g_{v}^{\vee}(\bflambda)=
  \begin{cases}
\displaystyle \sum_{j=0}^{r}\lambda_{j}\log \Big(
\frac{|\alpha_{j}|_{v}}{\lambda_{j}}\Big)& \mbox{if $v$ is Archimedean,}\\
\displaystyle  \max_{\bflambda}
\sum_{j=0}^{r}\lambda_{j}\log{|\alpha_{j}|_{v}} & \mbox{if $v$ is non-Archimedean.}
  \end{cases}
\end{displaymath}
For the Archimedean case, this formula follows from \cite[Proposition
5.8]{BPS:smthf}, whereas in the non-Archimedean case, it is given by
Example \ref{ex:canonicpullback0}. 

By \cite[Proposition~2.3.8(3)]{BPS14}, the $v$-adic roof function
$\vartheta_{\ol{D},v}$ is the direct image under the dual map
$H^{\vee}$ of the Legendre-Fenchel dual $g_{v}^{\vee}$, which gives
the stated formulae in~\eqref{eq:49}.
\end{proof}

\begin{definition}
  \label{def:12}
  Let $f\in \KK[M]$ be a Laurent polynomial and $X$ be a projective
  toric variety over $\KK$ given by a fan on $N_{\RR}$ that is
  compatible with the Newton polytope~$\newton(f)$.  Write
  $f=\sum_{j=0}^{r}\alpha_{j} \chi^{m_{j}}$ with $m_{j}\in M$ and
  $\alpha_{j}\in \KK^{\times}$.
The \emph{toric metrized divisor on $X$ associated to $f$} is defined as
\begin{displaymath}
  \ol{D}_{f}=\ol{D}_{\bfm,\bfalpha}, 
\end{displaymath}
the toric metrized divisor in \eqref{eq:46} for the data
$\bfm=(m_{0},\dots, m_{r})\in M^{r+1}$ and
$\bfalpha=(\alpha_{0},\dots, \alpha_{r})\in (\KK^{\times})^{r+1}$.  It
does not depend on the ordering of the terms of $f$. For $v\in \fM$,
we denote by $ \psi_{f,v}$ and $ \vartheta_{f,v}$ the $v$-adic metric
and roof functions of $\ol{D}_{f}$, respectively.
\end{definition}

\begin{lemma}
  \label{lemm:3}
    With notation as in Definition \ref{def:12}, the global section of
  $\cO(D_{f})$ associated to ${f}$ is $\ol{D}_{f}$-small.
\end{lemma}

\begin{proof}
  Set $\ol{D}=\ol{D}_{f}$ for short, and let $s=f s_{D}$ be the
  global section of $\cO(D)$ associated to ${f}$. For $v\in \fM$ and
  $p\in \TT(\ol{\KK}_{v})$,
\begin{displaymath}
  \|s(p)\|_{v}= |f(p)|_{v} \, \|s_{D}(p)\|_{v}= 
  \bigg| {\displaystyle
    \sum_{j=0}^{r}\alpha_{j}\chi^{m_{j}}( p)}\bigg|_v\|s_{D}(p)\|_{v}.
\end{displaymath}
It follows from \eqref{eq:48} that $\|s\|_{v}\le 1$ on
$\TT(\ol{\KK}_{v})$, and so $s$ is $\ol{D}$-small.
\end{proof}

The following result corresponds to Theorem \ref{thm:1} in the
introduction.

\begin{theorem}
\label{thm:4}
  Let $f_1,\ldots,f_n\in\KK [M]$, and let $X$ be a proper toric variety with
  torus $\TT_{M}$  and $\ol D_0$ a nef toric metrized divisor on $X$. Let $\Delta_{0}\subset M_{\RR}$ be the polytope of
  $D_{0}$ and, for $v\in \fM$, let
  $\vartheta_{0,v}\colon \Delta_{i}\to \RR$ be the $v$-adic roof function
  of $\ol D_{0}$. For $i=1,\dots, n$, let $\Delta_{i}\subset M_{\RR}$
  be the Newton polytope of $f_{i}$ and, for $v\in \fM$, let
  $\vartheta_{i,v}\colon \Delta_{i}\to \RR$ be the $v$-adic roof
  function on the metric associated to $f_{i}$.  Then
  \begin{equation*}
    \h_{\ol{D}_0}(Z(f_1,\ldots,f_n))
\leq \sum_{v\in\fM}n_v \MI(\vartheta_{0,v},\ldots,\vartheta_{n,v}).  
  \end{equation*}
\end{theorem}

\begin{proof}
  Let $\Sigma$ be the complete fan corresponding to the proper toric
  variety $X$. By taking a refinement, we can assume without loss of
  generality that $\Sigma$ is regular and compatible with the Newton
  polytopes $\Delta_{i}$, $i=1,\dots, n$.  Hence $X$ is a projective
  toric variety and $\ol{D}_{0}$  a nef toric metrized divisor, and
  there are nef toric Cartier divisors $D_{i}$, $i=1, \dots, n$,
  corresponding to these Newton polytopes.

  For $i=1,\dots, n$, we denote by $\ol{D}_{i}$ the toric metrized
  divisor associated to $f_{i}$ (Definition \ref{def:12}).  By
  Proposition~\ref{prop:10}, each $\ol{D}_{i}$ is semipositive and
  generated by small sections and, by Lemma~\ref{lemm:3}, the global
  sections $s_{i}$ of $\cO(D_{i})$ corresponding to $f_{i}$ are
  $\ol{D}_{i}$-small. Applying Corollary \ref{cor:2} and Theorem
  \ref{thm:heightoftoricvariety},
\begin{displaymath}
\h_{\ol{D}_{0}}\bigg(\prod_{i=1}^{n}\div(s_{i}) \bigg)\le \h_{\ol{D}_{0},\dots, \ol{D}_{n}}(X)
=\sum_{v\in\mathfrak{M}}
n_v\MI_M(\vartheta_{\ol{D}_0,v},\ldots,\vartheta_{\ol{D}_n,v}) .
\end{displaymath}
Due to Proposition \ref{prop:0cycles}\eqref{item:2}, the inequality $Z(f_1,\ldots,f_n)\leq
  \prod_{i=1}^n\div(s_i)$ holds. By the linearity of the global height
  and the nefness of $\ol{D}_{0}$, 
  \begin{displaymath}
\h_{\ol{D}_{0}}(   Z(f_1,\ldots,f_n)) \leq
\h_{\ol{D}_{0}}\bigg(     \prod_{i=1}^n\div(s_i)\bigg), 
  \end{displaymath}
which concludes the proof.  
\end{proof}

\begin{definition}\label{def:l1norm}
  Let $\bfalpha=(\alpha_0,\ldots,\alpha_r)\in(\KK^{\times})^{r}$ with
  $r\ge 1$. For  $v\in\fM$, the \emph{$v$-adic (logarithmic)
    length} of $\bfalpha$ is defined as
\[
\ell_v(\bfalpha)=
\begin{cases}
\log(\sum_{j=0}^{r}\abs{\alpha_j}_v) &\mbox{if $v$ is Archimedean,}\\
\log(\max_{0\leq j\leq r}\abs{\alpha_j}_v) &\mbox{if $v$ is non-Archimedean.}
\end{cases}
\]
The \emph{(logarithmic) length} of $\bfalpha$ is defined as
$\ell(\bfalpha)=\sum_{v \in\fM} n_{v} \ell_v(\bfalpha)$.

For a Laurent polynomial $f \in\KK[M]$, we define its \emph{$v$-adic
  (logarithmic) length}, denoted by $\ell_{v}(f)$, as the $v$-adic length of its
vector of coefficients, $v\in \fM$.  We also define its
\emph{(logarithmic) length}, denoted by $\ell(f)$, as the length of
its vector of coefficients.
\end{definition}

\begin{lemma}\label{lem:height2}
Let $\vartheta_{i}\colon \Delta_{i}\to \RR$ be
concave functions on convex bodies, $i=0,\dots, n$. Then 
\begin{displaymath}
\MI_M(\vartheta_{0},\ldots,\vartheta_{n})\leq
\sum_{i=0}^{n} \Big( \max_{x\in \Delta_{i}} \vartheta_{i}(x)\Big) \MV_M(\Delta_0,\ldots,\Delta_{i-1},\Delta_{i+1},\ldots,\Delta_n)
\end{displaymath}
\end{lemma}

\begin{proof}
  Set $c_i=\max_{x\in \Delta_{i}} \vartheta_{i}(x)$ for short. By the
  monotonicity of the mixed integral \cite[Proposition
  8.1]{PhilipponSombra:rBKe}
\begin{equation*}
\MI_M(\vartheta_{0},\ldots,\vartheta_{n})\leq\MI_M(c_0|_{\Delta_{0}},\ldots,c_n|_{\Delta_{n}}), 
\end{equation*}
where  $c_i|_{\Delta_{i}}$ denotes the constant function $c_{i}$ on
the convex body $\Delta_{i}$. By \cite[formula~(8.3)]{PhilipponSombra:rBKe}
\begin{equation*}
\MI_M(c_0|_{\Delta_{0}},\ldots,c_n|_{\Delta_{n}})
= \sum_{i=0}^{n} c_{i}\MV_M(\Delta_0,\ldots,\Delta_{i-1},\Delta_{i+1},\ldots,\Delta_n),
\end{equation*}
giving the stated inequality.
\end{proof}

\begin{corollary}\label{cor:height2}
With notation as in Theorem \ref{thm:4}, 
\begin{multline*}
\h_{\ol{D}_0}(Z(f_1,\ldots,f_n))
\leq
\bigg(\sum_{v\in\fM}n_{v} \max_{x\in
  \Delta_{0}}\vartheta_{0,v}(x)\bigg)\MV_{M}(\Delta_1,\ldots,\Delta_n)\\[-3mm]
+ \sum_{i=1}^n\ell(f_i)\MV_{M}(\Delta_0,\ldots,\Delta_{i-1},\Delta_{i+1},\ldots,\Delta_n).
\end{multline*}
In particular, for the canonical metric on $D_{0}$ (Example \ref{exm:6}),
\begin{equation}
\label{eq:21}
\h_{\ol{D}_0^{\mathrm{can}}}(Z(f_1,\ldots,f_n))
\leq
\sum_{i=1}^n\ell(f_i)\MV_{M}(\Delta_0,\ldots,\Delta_{i-1},\Delta_{i+1},\ldots,\Delta_n)\mbox{.}
\end{equation}
\end{corollary}

\begin{proof}
For  $1\le i\le n$ and $v\in \fM$, let $\vartheta_{i,v}$ be the $v$-adic roof
function of the toric semipositive metric associated to $f_{i}$. Using
\eqref{eq:49}, we compute the value of
$\psi_{i,v}(0)=-\vartheta_{i,v}^{\vee}(0)$ to obtain that
\begin{equation}
  \label{eq:54}
  \max_{x\in \Delta_{i}} \vartheta_{i,v}(x) = \ell_{v}(f_{i}).
\end{equation}

The first statement follows then from Theorem \ref{thm:4} and Lemma
\ref{lem:height2}.  The second statement is a particular case of the
first one, using the fact that the $v$-adic roof functions of
$\ol{D}^{\can}_{0}$ are the zero functions on $\Delta_{0}$.
\end{proof}

We readily derive from the previous corollary the following version of
the arithmetic B\'ezout theorem.
 
\begin{corollary}
\label{cor:3}
Let $f_{1},\dots, f_{n}\in \KK[x_{1},\dots, x_{n}]$ and let
$\ol{D}^{\can}$ be the divisor at infinity of $\PP^{n}_{\KK}$
equipped with the canonical metric. Then
\begin{displaymath}
\h_{\ol{D}^{\can}}(Z(f_1,\ldots,f_n)) \le
\sum_{i=1}^{n}\Big(\prod_{j\ne i}\deg(f_{j}) \Big) \ell(f_{i}),
\end{displaymath}
where $\deg(f_{i})$ denotes the total degree of the polynomial
$f_{i}$. 
\end{corollary}

\section{Comparisons, examples and applications} \label{sec:examples}

In this section, we first compare our main results (Theorem
\ref{thm:4} and Corollary~\ref{cor:height2}) with the previous
ones. Next, we compute the bounds given by these results in two
families of examples, and compare them with the actual height of the
$0$-cycles.  The first family of examples illustrates a case in which
these bounds do approach the height of the 0-cycle, while the second
one shows a situation where the bound of Theorem \ref{thm:4} is sharp
and that of Corollary \ref{cor:height2} is not. Finally, we present an
application bounding the height of the resultant of a 0-cycle defined
by a system of Laurent polynomials.

The first arithmetic analogue of the BKK theorem was proposed by
Maillot \cite[Corollaire 8.2.3]{Maillot:GAdvt}. With notations as in
Theorem \ref{thm:4}, suppose that $f_{1},\dots, f_{n}\in \ZZ[M]$ and
that $D_{0}$ is the nef toric divisor corresponding to the polytope
$\Delta_{0}=\sum_{i=1}^{n}\Delta_{i}$. Then Maillot's result amounts
to the upper bound
\begin{equation}
  \label{eq:8}
  \h_{\ol{D}_0^{\mathrm{can}}}(Z(f_1,\ldots,f_n))
  \leq
  \sum_{i=1}^n(\m(f_i)+L(\Delta_{i}))\MV_{M}(\Delta_0,\ldots,\Delta_{i-1},\Delta_{i+1},\ldots,\Delta_n),
\end{equation}
where $\m(f_{i})$ denotes the logarithmic Mahler measure of $f_{i}$,
and $L(\Delta_{i})$ a constant associated to the polytope
$\Delta_{i}$.

This result is similar to Corollary \ref{cor:height2} specialized to
a system of Laurent polynomials with integer coefficients, and the
toric divisor $D_{0}$ associated to the polytope given by the Minkowski sum 
 $\sum_{i=1}^{n}\Delta_{i}$, equipped with the canonical metric. The factors
$\m(f_i)+L(\Delta_{i})$ in \eqref{eq:8} and $\ell(f_{i})$ in
\eqref{eq:21} are comparable, albeit the fact that the constant
$L(\Delta_{i})$ is not effective, see \cite[Remark 4.2]{Sombra:msvtp}
for a discussion on this point. 

Another previous result in this direction was obtained by the second
author \cite[Th\'eor\`eme 0.3]{Sombra:msvtp}. Using again the notation
in Theorem \ref{thm:4}, suppose that $f_{1},\dots, f_{n}\in \ZZ[M]$
and that the polytope $\Delta_{0}$ associated to the nef toric divisor
$D_{0}$ contains $\Delta_{i}$, $i=1,\dots, n$. Then
\begin{equation*}
  \h_{\ol{D}_0^{\mathrm{can}}}(Z(f_1,\ldots,f_n))
  \leq n! \vol_{M}(\Delta_{0})   \sum_{i=1}^n \ell(f_i).
\end{equation*}
This result  is equivalent to the specialization of  the upper bound
in \eqref{eq:21} to 
a system of Laurent polynomials with integer coefficients and Newton
polytopes contained in the polytope $\Delta_{0}$. 

We next turn to the computation of the bounds given by Theorem
\ref{thm:4} and Corollary~\ref{cor:height2}  in two families of
examples. 

We keep the notation of \S\ref{sec:theorems}. We need the the
following auxiliary computation of mixed volumes. For its proof, we
recall that the mixed volume of a family of polytopes
$\Delta_{i}\subset \RR^{n}$, $i=1,\dots, n$, can be decomposed in
terms of mixed volumes of their lower dimensional faces as
\begin{equation}
  \label{eq:13}
  \MV_{n}(\Delta_{1},\dots, \Delta_{n})= -\hspace{-2mm}\sum_{u\in
    S^{n-1}}\Psi_{\Delta_{1}}(u) \MV_{n-1}(\Delta_{2}^{u},\dots,
  \Delta_{n}^{u}), 
\end{equation}
where $S^{n-1}$ is the unit sphere of $\RR^{n}$, $\Psi_{\Delta_{1}}$
is the support function of $\Delta_{1}$ as in \eqref{eq:12},
$\Delta_{i}^{u}$ is the unique face of $\Delta_{i}$ that minimizes the
functional $u$ on this polytope, and $\MV_{n}$ and $\MV_{n-1}$ denote
the mixed volume functions associated to the Lebesgue measure of
$\RR^{n}$ and $u^{\bot}\simeq\RR^{n-1}$, respectively. In fact, the sum ranges
through the normal vectors of the facets of each polytope. We refer to  \cite[formula
(5.1.22)]{Schneider:cbbmt} for more details.

\begin{lemma}\label{lem:examples}
  Let $\Delta\subset M_{\RR}$ be a lattice polytope, and $m_i\in M$,
  $i=2,\dots, n$, linearly independent lattice points.
  Denote by $\ol{0\,m_i}$ the segment between $0$ and $m_i$, and 
  $u\in N$  the smallest lattice point orthogonal to all the
  $m_{i}$'s,  which
  is unique up to a sign. Let  $P=
  \sum_{i=2}^{n}\ZZ m_{i}\subset M$ be the sublattice generated by the
  $m_{i}$'s, and  $P^{\mathrm{sat}}$  its saturation. Then
\[
\MV_M(\Delta,\ol{0\,m_2},\ldots,\ol{0\,m_{n}})=[{P}^{\mathrm{sat}}:P]\vol_{\ZZ}\langle \Delta,u\rangle,
\]
where $\langle \Delta,u\rangle $ is the image of $\Delta$ under the
functional $u\colon M_{\RR}\to \RR$.
\end{lemma}

\begin{proof}
  Choosing a basis, we identify $M=\ZZ^{n}$.  With this
  identification, $\MV_{M}=\MV_{n}$,  the mixed volume associated
  to the Lebesgue measure of $\RR^{n}$. The formula in~\eqref{eq:13}
  applied to the polytopes $\Delta,\ol{0\,m_2},\ldots,\ol{0\,m_{n}}$
  implies that
\begin{align}
\label{eq:33}
\nonumber
  \MV_{n}(\Delta,\ol{0\,m_2},\ldots,\ol{0\,m_{n}})=&
  -\Big(\Psi_{\Delta}\Big(\frac{u}{\|u\|}\Big)+\Psi_{\Delta}\Big(-\frac{u}{\|u\|}\Big)\Big)
                                                     \MV_{n-1}(\ol{0\,m_2},\ldots,\ol{0\,m_{n}})\\
=&
  -\frac{1}{\|u\|} (\Psi_{\Delta}({u})+\Psi_{\Delta}(-{u})) \MV_{n-1}(\ol{0\,m_2},\ldots,\ol{0\,m_{n}}),
\end{align}
where $\|u\|$ is the Euclidean norm. 
We have that 
\begin{equation}
  \label{eq:31}
  \Psi_{\Delta}({u})+\Psi_{\Delta}(-{u})= \min_{x\in \Delta}\langle
  x,u\rangle + \min_{x\in \Delta}\langle x,-u\rangle = -\vol_{\ZZ}\langle
  \Delta,u\rangle 
\end{equation}
By the Brill-Gordan duality theorem  \cite[Lemma
1]{Heath-Brown:dasn}, we have the equality
$\|u\|= \vol_{n-1}(P_{\RR}/P^{\mathrm{sat}})$, where $\vol_{n-1}$
denotes the Le\-bes\-gue measure of $u^{\bot}$. Hence
\begin{equation}
  \label{eq:32}
\frac{1}{\|u\|} \MV_{n-1}(\ol{0\,m_2},\ldots,\ol{0\,m_{n}})= \MV_{P^{\mathrm{sat}}}(\ol{0\,m_2},\ldots,\ol{0\,m_{n}})= 
[{P}^{\mathrm{sat}}:P].
\end{equation}
The result follows then from \eqref{eq:33}, \eqref{eq:31} and
\eqref{eq:32}. 
\end{proof}

\begin{example} \label{exm:2}
Let $d,\alpha\geq 1$ be integers and consider the system of Laurent
polynomials given by
\[
f_1=x_1-\alpha, \quad
f_2=x_2-\alpha x_1^d, \quad 
\dots , \quad f_n=x_n-\alpha x_{n-1}^d  \quad\in \QQ[x_{1}^{\pm1}, \dots, x_{n}^{\pm1}].
\]
Its zero set in $\TT_{\ZZ^{n}}=\GG_{\mathrm{m},\QQ}^{n}$ consists of
the  rational point
\[
p=(\alpha ,\alpha ^{d+1},\ldots,\alpha ^{d^{n-1}+d^{n-2}+\cdots+1})
\in \TT_{\ZZ^{n}}({\QQ})= ({\QQ}^{\times})^{n}.
\]
Let $X$ be a proper toric variety over $\QQ$, and
$\ol{D}^{\mathrm{can}}_0$  a nef toric Cartier divisor on $X$ equipped with
the canonical metric. 
Let  $\Delta_0\subset \RR^{n}$ be the polytope corresponding to~$D_{0}$ and, for $i=1,\ldots,n$, set 
\begin{displaymath}
  u_i=e_{i}+de_{i+1}+\dots+d^{n-i}e_{n}\in \ZZ^{n},
\end{displaymath}
where the $e_{j}$'s are the vectors in the standard basis of
$\ZZ^{n}$. The height of $p$ with respect to $\ol{D}_{0}^{\can}$ is 
\begin{equation}\label{eq:ex1-heightp}
\h_{\ol{D}^{\mathrm{can}}_0}( p)=\Big(\vol_\ZZ\Big\langle
\Delta_0,\sum_{i=1}^n u_i\Big\rangle\Big) \log(\alpha) .
\end{equation}
To prove this, let  $v\in\fM_\QQ$.   By \eqref{eq:35}, the local height of $ p$ with respect to the pair
$(\ol{D}_{0}^{\can},s_{D_{0}})$ is given by
\begin{displaymath}
\h_{\ol{D}^{\mathrm{can}}_0\hspace{-2mm},v}(p,s_{D_{0}})=-\log\|s_{D_{0}}(p)\|_{v,\can}=-\Psi_{\Delta_0}\big(\val_v(p)\big).
\end{displaymath}
Set $u=\sum_{i=1}^{n}u_{i}$ for short. Since $\val_v( p)=-\log\abs{\alpha
}_v\, u $, 
\begin{equation*}
-\Psi_{\Delta_0}\big(\val_v(p)\big)=
\begin{cases}
\log\abs{\alpha }_v\max\limits_{m\in\Delta_0\cap \ZZ^n}\langle m,u\rangle &\text{ if } v=\infty,\\[2mm]
\log\abs{\alpha }_v\min\limits_{m\in\Delta_0\cap \ZZ^n}\langle m, u\rangle &\text{ if } v\ne \infty.
\end{cases}
\end{equation*}
By adding these contributions, 
\begin{displaymath}
  \h_{\ol{D}^{\mathrm{can}}_0}( p)=  \log({\alpha }) \Big(
  \max\limits_{m\in\Delta_0\cap \ZZ^n}\langle m,u\rangle  -  \min\limits_{m\in\Delta_0\cap \ZZ^n}\langle m,u\rangle \Big),
\end{displaymath}
which gives the formula in \eqref{eq:ex1-heightp}.

Next we compare the value of the height of $ p$ with the bounds given
by Corollary~\ref{cor:height2}.  We have $\ell(f_i)=\log(\alpha +1)$
for all~$i$. Consider the dual basis of the $u_{i}$'s, given by
\begin{displaymath}
  m_1=e_1, m_2=e_2-de_{1}, \dots, m_{n}=e_{n}-de_{n-1}\in \ZZ^{n}.
\end{displaymath}
For $i=1,\dots, n$, the Newton polytope $\Delta_i$ of $f_{i}$ is a translate
of the segment~$\ol{0\,m_i}$, and $u_i$ is the smallest lattice point
in the line $(\sum_{j\neq i}\RR m_j)^\perp$.  Moreover the sublattice
$\sum_{j\neq i}\ZZ m_i$ is saturated. By Lemma~\ref{lem:examples}
\[
\MV_{\ZZ^n}(\Delta_0,\ldots,\Delta_{i-1},\Delta_{i+1},\ldots,\Delta_n)
=\vol_\ZZ\langle \Delta_0,u_i\rangle.
\]
Therefore, the bound given by Corollary~\ref{cor:height2} is
\begin{equation*}
\h_{\ol{D}_0^{\mathrm{can}}}( p)\leq 
\Big(\sum_{i=1}^n\vol_\ZZ\langle\Delta_0,u_i\rangle\Big) \log(\alpha+1).
\end{equation*}
Example \ref{exm:1} in the introduction consists of the particular
cases corresponding to the polytopes $\Delta_{0}=\Delta^{n}$, the
standard simplex of $\RR^{n}$, and
$\Delta_{0}= \conv(0, m_{1},\dots, m_{n})$.
\end{example}

In the following example, we exhibit a situation where the difference
between the bounds given by the results in \S\ref{sec:theorems} is
noticeable.  Recall that passing from Theorem~\ref{thm:4} to
Corollary~\ref{cor:height2} amounts to replacing the local roof
functions by constant functions on the polytope bounding them from
above.  Hence, to maximize the discrepancy between these two concave
functions, we look for local roof functions that are tent-shaped,
which is the situation where the difference between the mean value and
the maximum value of these functions is the greatest possible.

\begin{example}
Let $\alpha  \geq 1$ be an integer and consider the system of Laurent polynomials
\[
f_i=x_{i}-\alpha  \in \QQ[x_{1}^{\pm1}, \dots, x_{n}^{\pm1}], \quad
i=1,\dots, n,
\]
Its zero set in $\GG_{\mathrm{m},\QQ}^{n}$ is the rational point
$p=(\alpha ,\ldots,\alpha)\in (\QQ^{\times})^{n}$.  Let
$X=\PP^n_{\QQ}$ and let $\ol{E}^{\mathrm{can}}$ be the
divisor of the hyperplane at infinity equipped with the canonical
metric.  Then the height of $p$ with respect to
$\ol{E}^{\mathrm{can}}$ is
\begin{equation*}
\h_{\ol{E}^{\mathrm{can}}}(p)=\log (\alpha).
\end{equation*}

Next we compare the value of this height with the bound given by
Theorem \ref{thm:4}. Since the explicit computation of the mixed
integrals appearing in this bound is somewhat involved, instead of
giving its exact value we are going to approximate them with an upper
bound that is easier to compute.

The polytope associated to the toric Cartier divisor $E$ is
$\Delta_{0}= \Delta^{n}$, the standard simplex of $\RR^{n}$. For each
$v\in \fM_{\QQ}$, the $v$-adic roof function $\vartheta_{0,v}$ of
$\ol{E}^{\can}$ is the zero function on this simplex.

For each $i=1,\dots, n$, let $\Delta_{i}=\newton(f_{i})\subset\RR^{n}$
be the Newton polytope of $f_{i}$, which coincides with the segment
$\ol{0\, e_{i}}$.  For $v\in \fM_{\QQ}$, let $\vartheta_{i,v}$ be the
$v$-adic roof function associated to $f_{i}$ (Definition
\ref{def:12}).  This function is given, for
$t\, e_{i}\in \Delta_{i}=\ol{0\, e_{i}}$, by
\begin{align*}
\vartheta_{i,\infty}(t\, e_{i}) = 
  \begin{cases}
(1-t)\log(\alpha)      -t\log t -(1-t)\log(1-t)& \text{ if }
     v=\infty,\\
(1-t)\log\abs{\alpha}_v & \text{ if }
     v\ne \infty. 
  \end{cases}
\end{align*}
For the Archimedean place, the $v$-adic roof functions are
nonnegative, and so their mixed integral can be expressed as a mixed
volume
\begin{equation}
  \label{eq:22}
\MI_{\ZZ^{n}}(\vartheta_{0,\infty},\ldots,\vartheta_{n,\infty})
=\MV_{\ZZ^{n+1}}(\wt{\Delta}_0,\ldots,\wt{\Delta}_n),
\end{equation}
with $\widetilde{\Delta}_i=\conv\big(
\graph(\vartheta_{i,\infty}),\Delta_i\times\lbrace 0\rbrace)
\subset\RR^{n}\times\RR$. 
Consider the concave function 
 $\vartheta:\Delta^{n}\rightarrow\RR$  defined by
\[
\bfx=(x_{1},\dots, x_{n})\longmapsto\log (2)+\log( \alpha) \Big(1-\sum_{i=1}^n x_i\Big),
\]
and set
$\wt{\Delta}=\conv\big( \graph(\vartheta),\Delta^{n}\times\lbrace
0\rbrace\big)\subset\RR^{n}\times \RR$.
Notice that $\vartheta_{i,\infty}\leq\vartheta$ on $\Delta_{i}$, and so
 $\wt{\Delta}_i\subset\wt{\Delta}$, $i=0,\dots, n$. By the
monotony of the mixed volume,
\begin{multline} \label{eq:28}
\MV_{\ZZ^{n+1}}(\widetilde{\Delta}_0,\ldots,\widetilde{\Delta}_n)
\leq \MV_{\ZZ^{n+1}}(\widetilde{\Delta},\ldots,\widetilde{\Delta})
=(n+1)!\int_{\Delta^{n}}\vartheta \dd \bfx \\
=(n+1)!\Big(\log(2)\vol(\Delta^{n})+\log
(\alpha)\int_{\Delta^{n}}\sum_{i=1}^{n}x_i \dd \bfx\Big)
=(n+1)\log (2) +\log (\alpha)\mbox{.}
\end{multline}

When $v$ is non-Archimedean, we have that $|\alpha|_{v}\le 1$ because
$\alpha$ is an integer. Hence $\vartheta_{i,v}\leq 0$, and so the
mixed integral of these concave functions is nonpositive.  Theorem~\ref{thm:4} together with \eqref{eq:22} and \eqref{eq:28} gives the
upper bound
\begin{equation*}
\h_{\ol{E}^{\mathrm{can}}}(p)\leq (n+1)\log (2)+ \log (\alpha ).
\end{equation*}

To conclude the example, we compute the bound given by
Corollary~\ref{cor:height2}. For $i=1,\dots, n$, we have
that $\ell(f_i)=\log(\alpha+1)$ and
$\MV_{\ZZ^n}(\Delta_0,\ldots,\Delta_{i-1},\Delta_{i+1},\ldots,\Delta_n)=1$. Hence,
this bound reduces to
\begin{equation*}
\h_{\ol{E}^{\mathrm{can}}}(p)\leq n\log(\alpha+1),
\end{equation*}
 concluding the study of this example. 
\end{example}

As an application of our results, we bound the
size of the coefficients of the $\bfu$-resultant of the direct image
under an equivariant map of the $0$-cycle defined by a family of
Laurent polynomials. As in the previous sections,
let $(\KK,\fM)$ be an adelic field satisfying the product
formula,  $\ol{\KK}$ an algebraic closure of $\KK$,  and $M\simeq \ZZ^{n}$ a lattice.

\begin{definition}
\label{def:13}
  Let $W\in Z_0(\PP^r_\KK)$ be a 0-cycle of a projective space over
  $\KK$ and $\bfu=(u_0,\ldots,u_r)$ a set of $r+1$ variables.  Write
  $W_{\ol{\KK}}=\sum_{\bfq}\mu_\bfq \, \bfq \in Z_0\big(\PP^r_{\ol{\KK}}\big)$ for
  the $0$-cycle obtained from $W$ by the base change
  $\KK\hookrightarrow \ol{\KK}$. 
The \emph{$\bfu$-resultant} (or \emph{Chow form}) of $W$ is defined as
\[
\Res(W)=\prod_\bfq (q_0u_0+\cdots+q_ru_r)^{\mu_\bfq}\in\KK(\bfu)^{\times},
\]
the product being over the points
$\bfq=(q_0:\cdots:q_r)\in\PP_{\KK}^r(\ol{\KK})$ in the support of
$W_{\ol{\KK}}$. It is well-defined up to a factor in~$\KK^{\times}$.
\end{definition}

The length of a Laurent polynomial (Definition \ref{def:l1norm}) is
invariant under adelic field extensions and multiplication by scalars. It is also submultiplicative,
in the sense that it satisfies the inequality 
\begin{equation*}
  \ell(fg) \le \ell(f)+\ell(g),
\end{equation*}
for $f,g\in \KK[M]$. The following result corresponds to Theorem
\ref{thm:2} in the introduction.

\begin{theorem}
\label{thm:3}
  Let $f_1,\ldots,f_n\in\KK[M]$, $\bfm_{0}\in M^{r+1}$ and
  $\bfalpha_{0}\in (\KK^{\times})^{r+1}$ with $r\ge 0$. Set
  $\Delta_{0}=\conv(m_{0,0},\dots, m_{0,r})\subset M_{\RR}$ and let
  $\varphi:\TT_{M}\rightarrow\PP^r_\KK$ be the monomial map
  associated to $\bfm_{0}$ and $\bfalpha_{0}$ as in~\eqref{eq:42}. For
  $i=1,\dots, n$, let $\Delta_{i}\subset M_{\RR}$ be the Newton
  polytope of $f_{i}$, and $\bfalpha_{i}$ the vector of nonzero
  coefficients of $f_{i}$.  Then
\[
\ell(\Res(\varphi_*Z(f_1,\ldots,f_n)))
\leq
\sum_{i=0}^n
\MV_M(\Delta_0,\ldots,\Delta_{i-1},\Delta_{i+1},\ldots,\Delta_n) \, \ell (\bfalpha_{i}).
\]
\end{theorem}

\begin{proof}
Write $  Z(f_1,\ldots,f_n)_{\ol{\KK}}=\sum_{p} \mu_{p} \, p$, the sum being
over the points $p\in\TT_{M}(\ol{\KK})$. 
Since the length is  invariant 
under adelic field extensions and
submultiplicative, we deduce that
\begin{equation}
  \label{eq:47}
\ell(\Res(\varphi_*Z(f_1,\ldots,f_n)))
\le  \sum_{p} \mu_{p}\,  \ell \big(\alpha_{0,0}
\chi^{m_{0,0}}(p)\, u_0
+\cdots+ \alpha_{0,r} \chi^{m_{0,r}}(p)\, u_r\big).
\end{equation}

Let $X$ be a proper toric variety over $\KK$ defined by a fan that is
compatible with $\Delta_{i}$, $i=0,\dots, n$, and let $\ol{D}_{0}$ be
the toric metrized divisor on $X$ associated to $\bfm_{0}$ and
$\bfalpha_{0}$ as in \eqref{eq:46}. Given a point
$p\in\TT_{M}(\ol{\KK})$, we deduce from \eqref{eq:48} that
\begin{equation}
  \label{eq:52}
  \ell \big(\alpha_{0,0}
\chi^{m_{0,0}}(p)\, u_0
+\cdots+ \alpha_{0,r} \chi^{m_{0,r}}(p)\, u_r\big)=
\h_{\ol{D}_{0}}(p).
\end{equation}

By Proposition \ref{prop:10}, the toric metrized divisor is
semipositive and generated by small sections. In particular, it is nef.
Similarly as in \eqref{eq:54}, we also get from
Proposition~\ref{prop:10} that the $v$-adic roof functions of
$\ol{D}_{0}$ satisfy
$ \sum_{v\in\fM}n_v\max\vartheta_{0,v}=\ell(\bfalpha_{0})$.  Hence,
Corollary~\ref{cor:height2} implies that
\begin{equation}\label{eq:53}
\sum_{p}\mu_{p}\, \h_{\ol{D}_{0}}(p) \leq 
\sum_{i=0}^n\ell(\bfalpha_i)\MV(\Delta_0,\ldots,\Delta_{i-1},\Delta_{i+1},\ldots,\Delta_n).
\end{equation}
The statement follows then from \eqref{eq:47}, \eqref{eq:52} and
\eqref{eq:53}. 
\end{proof}


\providecommand{\bysame}{\leavevmode\hbox to3em{\hrulefill}\thinspace}
\providecommand{\MR}{\relax\ifhmode\unskip\space\fi MR }
\providecommand{\MRhref}[2]{%
  \href{http://www.ams.org/mathscinet-getitem?mr=#1}{#2}
}
\providecommand{\href}[2]{#2}

\end{document}